\documentclass[a4paper,11pt]{amsart}
\usepackage{fancyhdr}
\usepackage{amsmath}
\usepackage{dsfont}
\usepackage{hyperref}
\usepackage[mathscr]{eucal}
\usepackage[cp1251]{inputenc}
\usepackage[english]{babel}
\usepackage{cite,enumerate,float,indentfirst}
\usepackage[dvips]{graphicx}
\usepackage{xcolor}
\usepackage{latexsym,a4,mathrsfs,amsthm,amsmath,amssymb,url}
\usepackage{amsfonts}
\usepackage{amssymb}
\usepackage{epstopdf}
\usepackage{tikz}

\numberwithin{equation}{section}
\newtheorem{theorem}{Theorem}[section]
\newtheorem{lemma}[theorem]{Lemma}
\newtheorem{statement}[theorem]{Statement}

\theoremstyle{remark}
\newtheorem*{remark}{Remark}

\newcommand{\R}{\mathbb{R}}
\newcommand{\C}{\mathbb{C}}
\newcommand{\T}{\mathbb{T}}

\newcommand{\Cond}{{\bf (C0)}}

\newcommand{\X}{{\bf X}}
\newcommand{\A}{{\bf A}}
\newcommand{\D}{{\bf D}}
\newcommand{\M}{{\bf M}}

\newcommand{\EE}{{\bf E}}

\newcommand{\HH}{{\bf H}}
\newcommand{\OO}{{\bf O}}

\newcommand{\I}{{\bf I}}

\newcommand{\LL}{{\bf L}}
\newcommand{\V}{{\bf V}}
\newcommand{\U}{{\bf U}}
\newcommand{\J}{{\bf J}}
\newcommand{\W}{{\bf W}}
\newcommand{\Lam}{{\bf \Lambda}}
\newcommand{\Z}{{\bf Z}}
\newcommand{\Y}{{\bf Y}}

\newcommand{\ee}{{\bf e}}

\DeclareMathOperator{\Tr}{Tr}

\DeclareMathOperator{\E}{\mathbb{E}}
\DeclareMathOperator{\Var}{Var}

\DeclareMathOperator{\Pb}{\mathbb{P}}

\DeclareMathOperator{\re}{Re}
\DeclareMathOperator{\imag}{Im}

\begin{document}
\usetikzlibrary{calc,decorations.markings}

\vspace{1in}

\title[CLT for Products of Random Matrices]{\bf Distribution of Linear Statistics of Singular Values of the Product of  Random Matrices}

\author[F. G{\"o}tze]{F. G{\"o}tze}
\address{F. G{\"o}tze\\
 Faculty of Mathematics\\
 Bielefeld University \\
 Bielefeld, Germany
}
\email{goetze@math.uni-bielefeld.de}

\author[A. Naumov]{A. Naumov}
\address{A. Naumov\\
 Faculty of Computational Mathematics and Cybernetics\\
 Moscow State University \\
 Moscow, Russia
 }
\email{anaumov@cs.msu.su}

\author[A. Tikhomirov]{A. Tikhomirov}
\address{A. Tikhomirov\\
 Department of Mathematics\\
 Komi Research Center of Ural Division of RAS \\
 Syktyvkar, Russia
 }
\email{tichomir@math.uni-bielefeld.de}

\thanks{F. G{\"o}tze was supported by CRC 701 ``Spectral Structures and Topological
Methods in Mathematics'', Bielefeld University. A.Tikhomirov was supported by RFBR N~14-01-00500  and by Program of Fundamental Research Ural Division of RAS, Project 15-16-1-3. A.~Naumov was supported by RSCF~14-11-00196.}


\date{\today}

\begin{abstract}
In this paper we consider the product of two independent random matrices $\X^{(1)}$ and $\X^{(2)}$. Assume that $X_{jk}^{(q)}, 1 \le j,k \le n, q = 1, 2,$ are i.i.d. random variables with $\E X_{jk}^{(q)} = 0, \Var X_{jk}^{(q)} = 1$. Denote by $s_1, ... , s_n$ the singular values of $\W: = \frac{1}{n} \X^{(1)} \X^{(2)}$. We prove the central limit theorem for linear statistics of the squared singular values $s_1^2, ... , s_n^2$ showing  that the limiting variance depends on $\kappa_4: = \E (X_{11}^{1})^4 - 3$.

\end{abstract}

\maketitle

One of the main questions studied in Random Matrix Theory (RMT) is the asymptotic analysis of spectra of random matrices when the dimension goes to infinity.  For example it is well known since the pioneering work of Wigner~\cite{Wigner1958} that the empirical spectral distribution function weakly converges to the semicircle law. Another well known case is the sample covariance matrices $\W = \X \X^T$, where $\X$ is a matrix with independent entries, which was first studied in~\cite{MarchPastur1967} by Marchenko and Pastur. The distribution of singular values of products of random matrices with independent entries has been intensively studied, see for example~\cite{AlexGotzTikh2010b},~\cite{AlexGotzeTikh2011} and~\cite{AkemannIpsenKieb2013}.

All these results may be regarded as laws of large numbers for linear eigenvalue statistics. Thus fluctuations of such linear statistics of eigenvalues around its mean are of interest. There is a vast literature on this question. We mention the results of Jonsson~\cite{Jonsson1982}, Bai and Silverstain~\cite{BaiSilv2004} Sinai and Soshnikov~\cite{SinaSoshnikov1998}, Anderson and Zeitouni~\cite{AndZeitouni2006}, Lytova and Pastur~\cite{LytPastur2009}, Shcherbina~\cite{Shcherbina2011}, where the central limit theorem was proved.  The aim of this paper is to investigate the case of singular values of {\it products} of random matrices with independent entries. It will be shown that in this case the central limit theorem holds as well and the limiting variance will be explicitly determined.

\tableofcontents

\section{Introduction}

For any $m, n \ge 1$ we consider a family of independent real random variables $X_{j,k}^{(q)}$, $1 \le j,k \le n, q = 1, ... ,m$, defined on some probability space $(\Omega,\mathcal F,\Pb)$.
Assume that the following conditions \Cond~are fulfilled:\\
a) $X_{jk}^{(q)}$ are  identically distributed for $1 \le j, k \le n, q = 1, ... ,m$; \\
b) for any $1 \le j, k \le n$
$$
\E X_{j k}^{(q)} = 0 \text { and }  \E (X_{j k}^{(q)})^2 = 1;
$$
c) $\E (X_{jk}^{(q)})^4 = \mu_4 < \infty$.

The random variables $X^{(q)}_{jk}$ may depend on $n$, but for simplicity we shall not make this explicit in our  notations.

We introduce $m$ independent random matrices $\X^{(q)}, q = 1, ... , m$, as follows
$$
\X^{(q)} := \frac{1}{\sqrt n} [X_{jk}^{(q)}]_{j,k=1}^n.
$$
Denote by $s_1^2, ..., s_n^2$ the eigenvalues of the matrix $\W \W^T$, where $\mathbf W:= \prod_{q=1}^m\mathbf X^{(q)}$ and define the empirical spectral measure by
$$
F_n(x) = \frac{1}{n} \sum_{k = 1}^n \mathbb I(s_k^2 \le x).
$$
Here and in what follows $\mathbb I\{B\}$ denotes the indicator of the event $B$.

A fundamental problem in the theory of random matrices is to determine the limiting distribution of $F_n$ as the size of the random matrix tends to infinity.
It was shown by N.~Alexeev, F. G{\"o}tze and A.~Tikhomirov in~\cite{AlexGotzeTikh2011} that there exists a function $G_m(x)$ such that
\begin{equation}\label{eq: convergence_to_limit}
\lim_{n \rightarrow \infty} \sup_{x \in R} |\E F_n(x) - G_m(x)| = 0
\end{equation}
and $G_m(x)$ are defined by its moments $M_k, k \in \mathbb N$,
$$
M_k = \int_0^\infty x^k d G_m(x) = \frac{1}{m k + 1}\binom{k}{mk+k}
$$
which are so called Fuss-Catalan numbers. For $m= 1$ we get the well-known result of Marchenko-Pastur for sample covariance matrices.
The Fuss-Catalan numbers satisfy the following simple recurrence relation
$$
M_k = \sum_{k_0 + ... + k_m = k - 1} \prod_{\nu = 0}^m M_{k_\nu}.
$$
In~\cite{PensZycz2011} the density function $P_m(x)$ which satisfy
$$
\int_0^{K_m} x^k P_m(x) dx = M_k
$$
was found. Here $K_m: = (m+1)^{m+1}/m^m$. An explicit formula for $P_m(x)$ is given in Appendix~\ref{Fuss-Catalan Distribution}.

Denote by $s(z)$ the Stieltjes transform of the distribution $G_m(x)$
$$
s(z) = \int_{-\infty}^{\infty} \frac{1}{x - z}dG_m(x).
$$
It can be shown that $s(z)$ satisfy the following equation
$$
1 + z s(z) + (-1)^{m+1} z^m s(z)^{m+1} = 0.
$$
The result~\eqref{eq: convergence_to_limit} was proved under more general conditions then \Cond, it was assumed that the random variables may be non-identically distributed and satisfy a Lindeberg type condition
for the second moments, for details see~\cite{AlexGotzeTikh2011}. Under conditions \Cond~the result~\eqref{eq: convergence_to_limit} may be generalized and it can be shown
that $F_n$ weakly converges to $G_m$ in probability.  The latter may be rewritten in the following way
\begin{equation}\label{eq: LLN}
\int_{-\infty}^{\infty} f(\lambda) dF_n(\lambda) = \frac{1}{n}\sum_{k=1}^n f(s_k^2) \overset{p}{\longrightarrow}  \int_{-\infty}^{\infty} f(\lambda) dG_m(\lambda)
\end{equation}
which is valid for all continuous and bounded real functions $f(\lambda)$. We may interpret~\eqref{eq: LLN} as the law of large numbers. The natural question is to investigate a fluctuation of linear statistic
$$
S(\W): = \sum_{k=1}^n f(s_k^2)
$$
around its mean.

\subsection{Main result} Let $f(\lambda)$ be a smooth function with the Fourier transform given by
$$
\hat f(t) = \frac{1}{2\pi} \int_{-\infty}^\infty f(\lambda) e^{-it \lambda} d\lambda,
$$
We assume that $f(\lambda)$ satisfies the following condition
\begin{equation}\label{eq: condition of f}
\int_{-\infty}^\infty (1 + |t|^5) |\hat f(t)| dt < \infty.
\end{equation}
We will concentrate on the case of two random matrices, $m = 2$ and prove the following theorem which
is the main result of this paper
\begin{theorem} \label{th:main general case}
Let $m = 2$. Under conditions~\Cond and~\eqref{eq: condition of f} the statistic
$$
S^{(0)}: = S(\W) - \E S(\W)
$$
weakly converges to a Gaussian random variable $G$ with zero mean and variance given by
\begin{align}\label{eq: variance formula general case}
&\Var G = \frac{\kappa_4}{2} \left [ \int_{-a}^a f(\lambda^2) [p(\lambda) + \lambda p'(\lambda)] \, d\lambda \right ]^2 \\
&\qquad\qquad+\frac{1}{2\pi^2}  \int_{-a}^a \int_{-a}^a \frac{(f(\lambda^2) - f(\mu^2))^2}{(\lambda-\mu)^2} \nonumber\\
&\qquad\qquad\qquad\times \frac{[p(\lambda)-p'(\lambda)(\lambda-\mu)]}{3p(\mu)} \frac{[4p_1(\mu)^4 + 11 p_1(\mu)^2 + 4]}{4 p_1(\mu)^2 + 3}  \, d\lambda \, d\mu, \nonumber
\end{align}
where $\kappa_4 = \mu_4 - 3$, $p_1(\lambda) = \pi p(\lambda)$, $p(\lambda) = |\lambda| P_2(\lambda^2)$ is the symmetrized Fuss-Catalan density, and $a = \sqrt{K_2}$.
\end{theorem}
\begin{remark}
Obviously the result of Theorem~\ref{th:main general case} depends
on the distribution of $X_{ij}^{(q)}, 1 \le j,k \le n, q = 1, 2$ in terms of the fourth cumulant
rather than the second moment only.
This means that the limiting behaviour is not universal in the usual sense, a fact which typical for the central limit theorems of linear eigenvalue statistics.
\end{remark}
\begin{remark}
The result of Theorem~\ref{th:main general case} may be extended on the case when $X_{ij}^{(q)}, 1 \le j,k \le n, q = 1, 2$
are non-identically distributed. Here one has to impose additional assumptions, for example Lindeberg's condition
on the tails of fourth moments of $X_{jk}^{(q)}$, see Section~\ref{section: general case} for details.
\end{remark}
\begin{remark}
The case $m > 2$ is much more difficult to analyse. One may derive a formula for $Y(x,t)$ (see the definition below).
But it is not yet clear whether this expression is positive, due to the fact that the formula
for $P_m(x), m > 3$ is rather complicated. We plan to study this case in a subsequent paper.
\end{remark}

\subsection{Structure of the paper}
We divide the proof of Theorem~\ref{th:main general case} into two parts. In the section~\ref{section: gaussian case} we consider the Gaussian case and derive an analogue of Theorem~\ref{th:main general case}. Our method will be based on the result of Lytova and Pastur~\cite{LytPastur2009} and Tikhomirov~\cite{Tikh1980},\cite{Tikh2001}.
In the section~\ref{section: general case} we investigate the difference between the general case and the Gaussian case. Here we will use the methods of Bentkus, see~\cite{Bentkus2003} and Tikhomirov, see~\cite{Tikh1980},\cite{Tikh2001}. All auxiliary facts about Fuss-Catalan distribution, unitary matrix decomposition and its derivatives are collected in Appendix~\ref{Fuss-Catalan Distribution}-\ref{ap: laplace transform}.

\subsection{Applications} One motivation for investigating the asymptotic distribution of products of random matrices follows from recent applications in wireless telecommunication, see~\cite{Muller2002}. Compare as well very recent results in~\cite{ZhengWeiSpeicherMullerHamCor2015}. Other immediate applications are in finance~\cite{Bouchaud2007} and quantum entanglement~\cite{Collins2010},~\cite{Zyczkowski2011}.

\subsection{History}
There are many papers on the CLT for linear eigenvalue statistics of random matrices. We mention the results of Jonsson~\cite{Jonsson1982}, Bai and Silverstain~\cite{BaiSilv2004} Sinai and Soshnikov~\cite{SinaSoshnikov1998}, Anderson and Zeitouni~\cite{AndZeitouni2006}, Lytova and Pastur~\cite{LytPastur2009}, Shcherbina~\cite{Shcherbina2011}.
In our setting the result for $m = 1$  was derived by Lytova and Pastur in~\cite{LytPastur2009}. We will use their ideas in the proof of Theorem~\ref{th:main general case}. One may also find a lot of information about the CLT for linear eigenvalues statistics in the book of Bai and Silverstein~\cite{BaiSilv2010}. We also believe that one may apply the result of~\cite{BaiSilv2004} together with~\cite{Zheng2012} to derive Theorem~\ref{th:main general case} 
 when restricting oneself  to the class of {\it analytic} test functions with a different  (implicit) representation for the variance
using Cauchy integrals.

The distribution of singular values of products of random matrices with independent entries has been intensively studied in many papers. The relevant literature  is much to extensive in order to describe it here in detail, see for example~\cite{AlexGotzTikh2010b},~\cite{AlexGotzeTikh2011},~\cite{PensZycz2011}, ~\cite{AkemannIpsenKieb2013},~\cite{GotzKostersTikh2014} and very recent result~\cite{Forrester2015}. The central limit theorem for product of complex Ginibre matrices was derived by Breuer and Duits in~\cite{BreuerDu2013}. It is known that in the complex Ginibre case the squares of singular values of $\W$ form a determinantal point process and the joint density function is a bi-orthogonal ensemble, see~\cite{AkemannKieburgWei2013}.

\subsection{Notations}
In what follows we will use the following notations. Denote by $||\A||, ||\A||_2$ the operator and Hilbert-Schmidt
norms of $\A$ respectively. As usual $\Tr \A = \sum_{i=1}^n \A_{ii}$. We assume that all random variables are defined on a common probability space $(\Omega, \mathcal F, \Pb)$. By $\Var(\xi)$ we mean $\E\xi^2 - (\E \xi)^2$, where $\E$ is the mathematical expectation with respect to $\Pb$. By $C$ and $c$ we denote some constants which
do not depend on $n$. We introduce the symmetrized version of $f$, i.e.
$$
\widetilde f(x) =
\begin{cases}
f(x)  &\mbox{if } x \geq 0, \\
f(-x) & \mbox{if } x < 0.
\end{cases}
$$
By $*$ we denote the convolution operation, i.e. $f*g(t) = \int_0^t f(s) g(t-s) \, ds$.

\section{The Gaussian case} \label{section: gaussian case}
In this section we consider the special case when $X_{jk}^{(q)}, 1 \le j,k \le n, q = 1, 2$ has the Gaussian distribution. We change our notations of matrices and denote by $\Y^{(q)}, q = 1,2$ the matrix $\X^{(q)}$ with $X_{jk}^{(q)}$ replaced by the Gaussian random variables. The main result of this section is the following theorem.
\begin{theorem} \label{th:main}
Let $\Y^{(q)} = n^{-1/2}\{Y_{jk}^{(q)}\}_{j,k=1}^n, q = 1,2,$ be independent random matrices such that the random variables $Y_{jk}^{(q)}, j,k=1,...,n, q = 1, 2,$
satisfy the conditions $\Cond$. Then the statistic
$$
S^{(0)}: = S(\W) - \E S(\W)
$$
weakly converges to the Gaussian random variable $G$ with zero mean and variance given by
\begin{align}\label{eq: variance formula gaussian case}
&\Var[G] = \frac{1}{2\pi^2}  \int_{-a}^a \int_{-a}^a \frac{(f(\lambda^2) - f(\mu^2))^2}{(\lambda-\mu)^2} \nonumber\\
&\qquad\qquad\qquad\times \frac{[p(\lambda)-p'(\lambda)(\lambda-\mu)]}{3p(\mu)} \frac{[4p_1(\mu)^4 + 11 p_1(\mu)^2 + 4]}{4 p_1(\mu)^2 + 3}  \, d\lambda \, d\mu,
\end{align}
where $p_1(\lambda) = \pi p(\lambda)$, $p(\lambda) = |\lambda| P_2(\lambda^2)$ is the symmetrized Fuss-Catalan density, and $a = \sqrt{K_2}$.
\end{theorem}

\subsection{Symmetrization}

Before we start to prove Theorem~\ref{th:main} we will introduce and prove a simple Lemma.
Let $\xi^2$ be a positive random variable with the distribution function $F(x)$.
Define $\widetilde \xi:=\varepsilon\xi$, where $\varepsilon$ denotes a Rademacher random variable with $\Pb\{\varepsilon=\pm1\}=1/2$ which is independent of $\xi$.
Let $\widetilde F(x)$ denote the distribution function of $\widetilde \xi$. It satisfies the following equation
\begin{equation}\label{sym}
\widetilde F(x)=1/2(1+\text{sgn} \{x\}\,F(x^2)),
\end{equation}
\begin{lemma}\label{symmetrization}
For any one-sided distribution function $F(x)$ and $G(x)$ we have
\begin{equation*}
\sup_{x\ge 0}| F(x)-G(x)|=2\sup_x|\widetilde F(x)-\widetilde G(x)|,
\end{equation*}
where $\widetilde F(x)$ ($\widetilde G(x)$) denotes the symmetrization of $F(x)$ ($G(x)$ respectively) according to (\ref{sym}).
\end{lemma}
\begin{proof}By (\ref{sym}), we have for any $x\ge 0$
\begin{align*}
F(x)=2\widetilde F(\sqrt x)-1\\
G(x)=2\widetilde G(\sqrt x)-1.
\end{align*}
This implies
\begin{equation*}
\sup_{x\ge 0}| F(x)-G(x)|=2\sup_{x\ge 0}|\widetilde F(\sqrt x)-
\widetilde G(\sqrt x)|=2\sup_{x}|\widetilde F( x)-
\widetilde G( x)|.
\end{equation*}
Thus Lemma is proved.
\end{proof}
We apply this Lemma to the distribution of the  squared singular values of the matrix $\mathbf W$. Let us denote
\begin{align} \label{matrix H}
\HH^{(\nu)}=\left(\begin{matrix}&\Y^{(\nu)}& \OO\\
&\OO & {\Y^{(m-\nu+1)}}^T\end{matrix}\right) \text{ and }
\J:=\left(\begin{matrix}&\mathbf O  &  \I \\& \I &\mathbf O\end{matrix}\right).
\end{align}
For any $1\le a, b\le m$, put
\begin{equation}\label{matrix V}
 \V_{[a,b]}=
\begin{cases}
\prod_{k=a}^{b}\HH^{(k)},\quad\text{for }a\le b,\\
\I\quad\text{otherwise},
\end{cases}
\end{equation}
and $\V = \V_{[1,m]}, \hat \V = \V \J$.

Note that  $\hat \V$ is a symmetric matrix. The eigenvalues of the matrix $\hat \V$ are $-s_1,\ldots,-s_n,s_n,\ldots,s_1$. Note that the symmetrization of the distribution function $F_n(x)$
is a function $\widetilde{F}_n(x)$ which is the empirical distribution function  of the  eigenvalues of the matrix $\hat \V$. According to Lemma~\ref{symmetrization} we get
\begin{equation*}
\sup_x| \E F_n(x)-G_m(x)|= 2\sup_x|\E \widetilde F_n(x)-\widetilde G_m(x)|,
\end{equation*}
and~\eqref{eq: LLN} may be rewritten as follows
\begin{equation*}
\int_{-\infty}^{\infty}  \widetilde f(x) d \widetilde F_n(x) = \frac{1}{2n}\sum_{k=1}^n [f(s_k) + f(-s_k)] \overset{p}{\longrightarrow} \int_{-\infty}^{\infty}  \widetilde f(x) d \widetilde G_m(x).
\end{equation*}

It is straightforward to check that the Stieltjes of $ \widetilde G(x)$ satisfies the following equation
\begin{equation}\label{eq: St transfrom of G_m}
1+zs(z)+(-1)^{m+1}z^{m-1}s^{m+1}(z) = 0.
\end{equation}

To finish our linearization we mention that
$$
\int_{0}^\infty f(x) d F_n(x) = \int_{-\infty}^{\infty}  \widetilde f(x^2) d \widetilde F_n(x).
$$
This means that we may substitute $f(x)$ by $f(x^2)$ and consider its symmetrization $\widetilde f(x)$.
In what follows we will consider symmetrized distribution functions only and omit the symbol $"\,\widetilde {\cdot}\,"$ in the corresponding notation.
In the new notations we will have
$$
S^0 = \frac{1}{2} [\Tr  f(\hat \V) - \E \Tr  f(\hat \V)].
$$

\subsection{Empirical Poincar\'e Inequalities}
Following~\cite{BobGotzTikh2010} we say that a probability measure $\mu$ on $\R^d$ satisfies a Poincar\'e-type inequality with constant $\sigma^2$
if for any bounded smooth function $g$ on $\R^d$ with gradient $\bigtriangledown g$,
\begin{equation}\label{eq: PI}
\Var(g) \le \sigma^2 \int |\bigtriangledown g|^2 d \mu,
\end{equation}
where $\Var(g) = \int g^2 d\mu - (\int g d\mu)^2$. In this case we write $PI(\sigma^2)$ for short.

Assume that the random variables $X_1, ... , X_n$ have a joint distribution $\mu$ on $\R^n$, satisfying the Poicare-type inequality~\eqref{eq: PI}. Given a bounded smooth complex-valued function $f$ on the real line,
one may apply~\eqref{eq: PI} to
$$
g(x_1, ... , x_n) = \frac{f(x_1) + ... + f(x_n)}{n} = \int f(x) dF_n(x),
$$
where $F_n$ in the empirical measure, defined for observation $X_1 = x_1, ... , X_n = x_n$. Since
$$
|\bigtriangledown g|^2 = \frac{|f'(x_1)|^2 + ... + |f'(x_n)|^2}{n^2} = \frac{1}{n}\int |f'(x)|^2 dF_n(x),
$$
we obtain the following statement, see~\cite{BobGotzTikh2010}[Proposition~4.3],
\begin{statement}\label{stat: PI}
Under $PI(\sigma^2)$, for any smooth $F$-integrable function $f: \R \rightarrow \C$,
$$
\E \left | \int f(x) d F_n(x) - \int f(x) d F(x) \right |^2 \le \frac{\sigma^2}{n} \int |f'(x)|^2 d F(x),
$$
where $F(x): = \E F_n(x)$.
\end{statement}
We will use the following linearization trick from~\cite{Burda2011}. Let us consider the matrix $\hat \V = [\prod_{j=1}^{m-1} \HH^{(j)} ] \HH^{(m)} \J$. We form the following $mn\times mn$  matrix
\begin{align*}
\M =\left [
\begin{matrix}
\OO & \HH^{(1)} & \OO & \OO & ... & \OO \\
\OO & \OO & \HH^{(2)} & \OO & ... & \OO \\
    &     &          &     & ... &     \\
\OO & \OO & \OO & \OO & ... & \HH^{(m-1)}\\
\HH^{(m)}\J &\OO & \OO & \OO & ... & \OO
\end{matrix} \right ]
\end{align*}
Then the $m$-th power of $\M$ is a diagonal block matrix, there the first block is equal to $\hat \V$, the second -- $\HH^{(2)} \HH^{(3)} ... \HH^{(m)}\J \HH^{(1)}$ and so on.
The eigenvalues of $\M^m$ are the eigenvalues of $\hat \V$ with multiplicity $m$. We denote the eigenvalues of $\M$ by $\lambda_1, ... , \lambda_{mn}$ and their empirical distribution function by
$G_n(\lambda)$. Then we have for an even function $f$
$$
\int f(x) d F_n(x) = \frac{1}{n}\sum_{j=1}^{n} f(s_j) = \frac{1}{2nm} \sum_{j=1}^{2mn} f(\lambda_i^m) = \int f(\lambda^m) d G_n(\lambda).
$$
Without loss of generality we assume that $\lambda_1, ... , \lambda_n$ are real positive eigenvalues $s_1^{1/m}, ... , s_n^{1/m}$.
All other eigenvalues may be derived by a rotation on an angle $\theta_k = \frac{k \pi}{m}, k = 1 , ... , 2m-1$. Let $\theta_0 = 0$. We denote the empirical spectral distribution of
$e^{i\theta_k}\lambda_1, ... , e^{i\theta_k}\lambda_n$ by $G_{n, k}$. It is easy to see that
\begin{equation}\label{eq: from singlar values to eigenvalues of M}
\int f(\lambda^m) d G_n(\lambda) = \frac{1}{2m}\sum_{k=0}^{2m-1} \int_{\T_k} f(\lambda^m) d G_{n,k}(\lambda),
\end{equation}
where $\T_k = e^{i \theta_k} \R$.

The joint distribution $P$ of the collection $\{Y_{jk}^{(q)} ,j,k=1,..., n, q = 1, ... , m\}$ represents a product probability measure
on the Euclidian space $\R^N$ of dimension $N = m n^2$,
while the joint distribution $\mu$ of the spectral values $\lambda_1, ..., \lambda_n$ is a probability measure
on $\R^{n}$, obtained from $P$ as the image under the map $T = P \cdot S$, where $S$ is the map from matrices to their eigenvalues and $P$ is the projector on the subspace of the dimension $n$.
We will apply the following Lemma (see~\cite{BobGotzTikh2010}[Lemma~7.1]
\begin{lemma}
Let $\mu_1, ... , \mu_N$ be probability measures on $\R$, satisfying $PI(\sigma^2)$.
The image of the product measure $P = \mu_1 \otimes  ... \otimes  \mu_N$ under any Lipshitz map $T: \R^N \rightarrow \R^n$ satisfies
$PI(\sigma^2 ||T||_{Lip}^2)$, where
$$
||g||_{Lip}: = \sup_{x \neq y} \frac{\rho_2(g(x), g(y))}{\rho_1(x,y)}
$$
and $\rho_1, \rho_2$ are metrics in $\R^N$ and $\R^n$ respectively.
\end{lemma}
In our case
$$
\sum_{j = 1}^n |\lambda_j - \lambda_j'|^2 \le ||\M - \M'||_2^2 = \frac{2}{n}\sum_{q=1}^m \sum_{j,k=1}^{n} |Y_{jk}^{(q)} - (Y_{jk}^{(q)})'|^2
$$
and $||T||_{Lip} = \frac{\sqrt 2}{\sqrt n}$. Since $Y_{jk}^{(q)}$ satisfies $PI(\sigma^2)$ it follows from~\eqref{eq: from singlar values to eigenvalues of M} and Statement~~\eqref{stat: PI} that
\begin{equation}\label{eq:variance EPI}
\E \left | \int f(x) d F_n(x) - \int f(x) d F(x) \right |^2 \le \frac{\sigma^2 m^2}{n^2} \int |x|^{\frac{2m-2}{m}}|f'(x)|^2 \, d F(x).
\end{equation}

\subsection{Proof of CLT in the Gaussian case}
In this subsection we give the proof of Theorem~\ref{th:main}.

\begin{proof}[Proof of Theorem~\ref{th:main}]
Let us denote the characteristic function of $S^0$ by $Z_n(x)$, i.e.
$$
Z_n(x) := \E e^{i x S^{0}}.
$$
To prove Theorem~\ref{th:main} it is sufficient to derive that
$$
\lim_{n \rightarrow \infty} Z_n(x) = Z(x),
$$
where $Z(x)$ is a characteristic function of the Gaussian random variable $G$ with zero mean and variance given by the formula~\eqref{eq: variance formula gaussian case}, i.e.,
$$
Z(x) = \E e^{it G} = e^{-\Var[G] t^2/2}.
$$
One has to show that $Z(x) = 1 - \Var[G] \int_0^x y Z(y) \, dy$.  Since
$$
Z_n(x) = 1 + \int_0^x Z'(y) \, dy,
$$
similarly to Lytova and Pastur(\cite{LytPastur2009}) it is sufficient to prove that any converging subsequences $\{Z_{n_l}\}$ and $\{Z_{n_l}'\}$ satisfy
\begin{equation}\label{eq: we have to show}
\lim_{n_l \rightarrow \infty} Z_{n_l}(x) = Z(x), \quad \lim_{n_l \rightarrow \infty} Z_{n_l}'(x) = - x Z(x) \Var[G].
\end{equation}
and show that $\Var[G]$ is given by the formula~\eqref{eq: variance formula gaussian case}.

Taking derivative of $Z_n(x)$ we get
$$
Z_n'(x) = i \E S^0 e^{ixS^0} = \frac{i}{2} \int_{-\infty}^\infty \hat f(t) \E [\Tr \U(t) - \E \Tr \U(t)] e^{i x S^0} \, dt,
$$
where we have applied the Fourier inverse formula
$$
f(\lambda) = \int_{-\infty}^{\infty} \hat f(t) e^{it\lambda} \, dt.
$$
We introduce further notations
\begin{align*}
&u_n(t) : = \frac{1}{2} \Tr \U(t), \\
&u_n^0(t) : = u_n(t) - \E u_n(t), \\
&Y_n(x, t) : = \E u_n^0(t) e^{i x S^0}.
\end{align*}
In these notations we may rewrite $Z_n'(x)$ as follows
\begin{equation}\label{eq: Z_n derivative equation}
Z_n'(x) = i  \int_{-\infty}^\infty \hat f(t) Y_n(x,t) \, dt.
\end{equation}
From unitary matrix representation~\eqref{eq: U matrix representation} it follows that
$$
u_n(t) = \sum_{j=1}^n U_{jj}(t) = \sum_{j=1}^n U_{j+n, j+n}(t).
$$
The following Lemma gives the estimates for the variance of $u_n(t)$, its derivative $u_n'(t)$
with respect to the argument $t$, and $Y_n(x,t)$.
\begin{lemma}\label{l: trace variance}
Under condition of Theorem~\ref{th:main} we have
\begin{align*}
\Var(u_n(t)) \le C_1 t^2, \quad \Var(u_n'(t)) \le C_2(1+ t^2), \quad |Y_n(x,t)| \le \sqrt{C_1} t.
\end{align*}
\end{lemma}
\begin{proof}
The statement of this Lemma for $u_n(t)$ and $u_n'(t)$ follows from~\eqref{eq:variance EPI} applied to $f(x) = \cos(tx)$ and
$f(x) = -x \sin(tx)$ respectively.
From the Cauchy-Schwarz inequality we conclude that
$$
|Y_n(x,t)| = |\E((u_n(t) - \E u_n(t)) e^{ixS^0})| \le \Var^{1/2}(u_n(t)) \le \sqrt{C_1} t.
$$
\end{proof}
From Lemma~\ref{l: trace variance} we may conclude that
$$
\left|\frac{\partial Y_n(x,t)}{\partial t}\right| \le \Var^{1/2}(u_n'(t)) \le C_2^{1/2} \sqrt{1+t^2}
$$
and
$$
\left|\frac{\partial Y_n(x,t)}{\partial x}\right| \le \Var^{1/2}(u_n(t))\Var^{1/2}(S_0) \le C_1^{1/2} t \sup_{\lambda \in \R} f'(\lambda)
$$
One may see that $Y_n(x,t)$ is bounded and equicontinues on any finite set of $\R^2$. Similarly to Lytova and Pastur it is sufficient to show that any uniformly converging subsequence of $\{Y_n\}$ has the same limit $Y$, leading to~\eqref{eq: we have to show}.

\subsection{Product of two random square matrices}
Let $m = 2$. We investigate the quantity $Y_n(x,t)$. Applying Duhamel's formula
$$
\U(t) = \I + i \int_0^t \hat \V \U(s) \, ds
$$
we will have
$$
Y_n(x,t) = \frac{i}{2} \int_0^t \E [ \Tr \hat \V \U(s) - \E \Tr \hat \V \U(s) ] e^{i x S^0} ds = \frac{1}{2} \mathcal A_1 + \frac{1}{2} \mathcal A_2,
$$
where
$$
\mathcal A_1 = \frac{i}{\sqrt n} \int_0^t \sum_{j, k = 1}^n \E  [ Y_{jk}^{(1)} [\HH^{(2)} \J \U(s)]_{k, j} - \E Y_{jk}^{(1)} [\HH^{(2)} \J \U(s)]_{k, j} ] e^{i x S^0} \, ds
$$
and
\begin{align*}
\mathcal A_2 = \frac{i}{\sqrt n} \int_0^t \sum_{j, k = 1}^n &\E [ Y_{jk}^{(2)} [\HH^{(2)} \J \U(s)]_{j+n, k+n } \\
&- \E  Y_{jk}^{(2)} [\HH^{(2)} \J \U(s)]_{j+n, k+n } ]e^{i x S^0}  \, ds.
\end{align*}
Let us consider the term $\mathcal A_1$. Applying $\E \xi f(\xi) = \E \xi^2 \E f'(\xi)$ valid for Gaussian random variable $\xi$ with zero mean, we get
\begin{equation}\label{eq: mathcal A representation}
\mathcal A_1 = I_1 + I_2,
\end{equation}
where
$$
I_1 = \frac{i}{\sqrt n} \int_0^t \sum_{j, k = 1}^n \E  \left [\left [\HH^{(2)} \J \frac{\partial\U(s)}{\partial Y_{jk}^{(1)}} \right ]_{k, j}
-\E \left [\HH^{(2)} \J \frac{\partial\U(s)}{\partial Y_{jk}^{(1)}} \right]_{k, j} \right ] e^{i x S^0} \, ds
$$
and
$$
I_2 = -\frac{x}{\sqrt n} \int_0^t \sum_{j, k = 1}^n \E [\HH^{(2)} \J \U(s)]_{k, j } \frac{\partial S}{\partial Y_{jk}^{(1)} } e^{i x S^0} \, ds.
$$
From Lemma~\ref{l: U derivative} it follows that
\begin{align*}
&\sum_{j,k=1}^n \left [\HH^{(2)} \J \frac{\partial\U(s)}{\partial Y_{jk}^{(1)}} \right ]_{k, j}  \\
&\qquad\qquad = \frac{i}{\sqrt n} \sum_{j,k = 1}^n \sum_{l=1}^{2n}  [\HH^{(2)} \J]_{k,l}  [\U \HH^{(1)} ]_{l,k+n}*[\U]_{j, j} (s) \\
&\qquad\qquad +\frac{i}{\sqrt n} \sum_{j,k = 1}^n \sum_{l=1}^{2n}  [\HH^{(2)} \J]_{k,l} [\U \HH^{(1)} ]_{j,k+n}*[\U]_{l, j} (s) \\
&\qquad\qquad= \frac{i}{\sqrt n} \int_0^s u_n(s - s_1) \sum_{k = 1}^n [\HH^{(2)} \J \U(s_1) \HH^{(1)} ]_{k, k+n}  \, ds_1 \\
&\qquad\qquad+\frac{i}{\sqrt n} \int_0^s \sum_{j, k = 1}^n [\HH^{(2)} \J \U(s_1) ]_{k, j} [\U(s - s_1) \HH^{(1)} ]_{j,k+n}   \, ds_1.
\end{align*}
Let us denote
$$
t_n(s): = \sum_{k=1}^n [\HH^{(2)} \J \U(s) \HH^{(1)} ]_{k, k+n}.
$$
In these notations we may write, applying Lemma~\ref{l: second trace variance},
$$
I_1 = - \frac{1}{n} \int_0^t \, ds \int_0^s \E[ u_n(s-s_1) t_n(s_1) - \E u_n(s-s_1) t_n(s_1)]e^{i x S^0} \, ds_1 + r_n(t),
$$
where
$$
|r_n(t)| \le C\frac{t^3}{\sqrt n}.
$$
In what follows for simplicity we will not specify the term $r_n(t)$, but one should have in mind that $r_n(t)$ goes to
zero as $n$ goes to infinity.  Let us rewrite the difference $\E [ t_n(s_1) u_n (s - s_1) - \E t_n(s_1) u_n (s - s_1) ]$. We have
\begin{align}\label{eq: u_n representation m = 2 first}
t_n(s_1) u_n (s - s_1)& = t_n^0(s_1) u_n^0(s- s_1) + t_n^0(s_1) \E u_n(s - s_1) \\
&+ u_n^0(s-s_1) \E t_n(s_1) + \E t_n(s_1) \E u_n(s-s_1). \nonumber
\end{align}
and
\begin{align}\label{eq: u_n representation m = 2 second}
t_n(s_1) u_n (s - s_1) &- \E t_n(s_1) u_n (s - s_1)  =  t_n^0(s_1) u_n^0(s- s_1)\\
&+ t_n^0(s_1) \E u_n(s - s_1) + u_n^0(s-s_1) \E t_n(s_1) \nonumber \\
&-\E t_n^0(s_1) u_n^0 (s - s_1) . \nonumber
\end{align}
From~\eqref{eq: u_n representation m = 2 second} the term $I_1$ may be rewritten as follows
\begin{align*}
I_1 &= - \frac{1}{n} \int_0^t \, ds \int_0^s \E u_n(s - s_1) \E t_n^0(s_1) e^{i x S^0} \, ds_1 \\
&- \frac{1}{n} \int_0^t \, ds \int_0^s \E t_n(s - s_1) Y_n(x, s_1) \, ds_1 + r_n(t) =: I_{11} + I_{12} + r_n(t).
\end{align*}
Let us investigate $t_n(s)$.
We may write
\begin{align*}
\E t_n(s) &= \frac{1}{\sqrt n} \sum_{j,k = 1}^n \E Y_{jk}^{(2)} [\U(s) \HH^{(1)} ]_{k+n,j+n} \\
&= \frac{1}{\sqrt n} \sum_{j,k = 1}^n \E \left [ \frac{\partial \U(s)}{\partial Y_{jk}^{(2)}} \HH^{(1)}  \right ]_{k+n,j+n} +
\frac{1}{\sqrt n} \sum_{j,k = 1}^n \E \left [ \U(s) \frac{\partial \HH^{(1)}}{\partial Y_{jk}^{(2)}} \right ]_{k+n,j+n}.
\end{align*}
From Lemma~\ref{l: U derivative q geq 2} we conclude
\begin{align}\label{eq: t_n calculation}
&\sum_{j,k=1}^n\left [ \frac{\partial \U(s)}{\partial Y_{jk}^{(2)}} \HH^{(1)}  \right ]_{k+n,j+n} \\
& = \frac{i }{\sqrt n} \int_0^s \sum_{j,k = 1}^n   \sum_{l=1}^{2n} [\U(s_1) ]_{k+n,k+n}[ \HH^{(2)} \J\U(s-s_1)]_{j+n, l} [\HH^{(1)}]_{l, j+n} \, ds_1 \nonumber\\
&+\frac{i }{\sqrt n} \int_0^s \sum_{j,k = 1}^n  \sum_{l=1}^{2n}   [\U(s_1) ]_{l,k+n} [ \HH^{(2)} \J\U(s-s_1)]_{j+n, k+n} [\HH^{(1)}]_{l, j+n} \, ds_1 \nonumber\\
&= \frac{i }{\sqrt n} \int_0^s u_n(s_1)  \sum_{j = 1}^n [ \HH^{(2)} \J\U(s-s_1) \HH^{(1)}]_{j+n, j+n} \, ds_1 \nonumber\\
&+\frac{ i s }{\sqrt n} \sum_{k = 1}^n   [\U(s) \hat \V]_{k+n, k+n}.\nonumber
\end{align}
It is easy to see that
\begin{align*}
&\sum_{j = 1}^n [ \HH^{(2)} \J\U(s-s_1) \HH^{(1)}]_{j+n, j+n} = \sum_{j = 1}^n [ \hat \V \U(s-s_1)]_{j+n, j+n} \\
&= -i \sum_{j = 1}^n [\U'(s-s_1)]_{j+n, j+n} = - i u_n'(s-s_1),
\end{align*}
and
$$
\sum_{k = 1}^n   [\U(s) \hat \V]_{k+n, k+n} = -i u_n'(s).
$$
For the second term we have
\begin{align*}
\sum_{j,k = 1}^n \E \left [ \U(s) \frac{\partial \HH^{(1)}}{\partial Y_{jk}^{(2)}} \right ]_{k+n,j+n} = \sqrt n \E u_n(s).
\end{align*}
It follows that
$$
\E t_n(s) = \frac{1}{n} \int_0^s \E u_n(s_1) u_n'(s-s_1) \, ds_1 + \E u_n(s) + \frac{s}{n} \E u_n'(s).
$$
and
\begin{align*}
I_{12} &= -\frac{1}{n^2}  \int_0^t \, ds \int_0^s  Y_n(x, s_1) \, ds_1  \int_0^{s-s_1} \E u_n(s_2) u_n'(s - s_1 - s_2) \, ds_2  \\
& - \frac{1}{n} \int_0^t \, ds \int_0^s  Y_n(x, s_1) \E u_n(s-s_1) \, ds_1 \\
& - \frac{1}{n^2}\int_0^t \, ds \int_0^s  Y_n(x, s_1) \E u_n'(s - s_1) \, ds_1.
\end{align*}
Since $|Y_n(x,s)| \le C$ (see Lemma~\ref{l: trace variance})  and $\E |u_n'(s-s_1)| \le n \sqrt n$ we get
\begin{align*}
I_{12} &= -\frac{1}{n^2}  \int_0^t \, ds \int_0^s  Y_n(x, s_1) \, ds_1  \int_0^{s-s_1} \E u_n(s_2) u_n'(s - s_1 - s_2) \, ds_2  \\
& - \frac{1}{n} \int_0^t \, ds \int_0^s  Y_n(x, s_1) \E u_n(s-s_1) \, ds_1 + r_n(t)
\end{align*}
Applying Lemma~\ref{l: trace variance} we may write
\begin{align*}
I_{12} &= -\frac{1}{n^2}  \int_0^t \, ds \int_0^s  Y_n(x, s_1) \, ds_1  \int_0^{s-s_1} \E u_n(s_2) \E u_n'(s - s_1 - s_2) \, ds_2  \\
& - \frac{1}{n} \int_0^t \, ds \int_0^s  Y_n(x, s_1) \E u_n(s-s_1) \, ds_1 + r_n(t).
\end{align*}
Changing the limits of integration we get
\begin{align*}
I_{12} = -\frac{1}{n^2}  \int_0^t Y_n(x, s) \, ds \int_0^{t-s} \E u_n(s_1) \E u_n(t - s - s_1) \, ds_1 + r_n(t).
\end{align*}

We investigate now $\E t_n^0(s) e^{i x S^0}$.
\begin{align*}
&\E t_n^0(s) e^{i x S^0} \\
&=\frac{1}{\sqrt n} \sum_{j,k = 1}^n \E [ Y_{jk}^{(2)} [\U(s) \HH^{(1)} ]_{k+n,j+n} - \E  Y_{jk}^{(2)} [\U(s) \HH^{(1)} ]_{k+n,j+n}] e^{i x S^0} \\
& =\frac{1}{\sqrt n} \sum_{j,k = 1}^n \E \left [ \left [ \frac{\partial \U(s)}{\partial Y_{jk}^{(2)}} \HH^{(1)}  \right ]_{k+n,j+n} -
\E \left [ \frac{\partial \U(s)}{\partial Y_{jk}^{(2)}} \HH^{(1)}  \right ]_{k+n,j+n} \right ] e^{i x S^0} \\
&+\frac{1}{\sqrt n} \sum_{j,k = 1}^n \E \left [ \left [ \U(s) \frac{\partial \HH^{(1)}}{\partial Y_{jk}^{(2)}} \right ]_{k+n,j+n} -
\E \left [ \U(s) \frac{\partial \HH^{(1)}}{\partial Y_{jk}^{(2)}} \right ]_{k+n,j+n} \right ] e^{i x S^0} \\
&+ \frac{i x}{\sqrt n} \sum_{j,k = 1}^n \E [\U(s) \HH^{(1)} ]_{k+n,j+n} \frac{\partial S}{\partial Y_{jk}^{(2)}} e^{i x S^0} =: \mathcal J_1 + \mathcal J_2 + \mathcal J_3.
\end{align*}
For the first term $\mathcal J_1$ we may use~\eqref{eq: t_n calculation} and get
\begin{align*}
\mathcal J_1 &= \frac{1}{n} \int_0^s \E [ u_n(s_1) u_n'(s-s_1) - \E u_n(s_1) u_n'(s-s_1)] e^{i x S^0} \, ds_1 \\
&+\frac{s}{n} \E[u_n'(s) - \E u_n'(s)] e^{i x S^0}.
\end{align*}
Repeating the step~\eqref{eq: u_n representation m = 2 first} and~\eqref{eq: u_n representation m = 2 second} the last relation may be rewritten in the following way
\begin{align*}
\mathcal J_1 & = \frac{1}{n} \int_0^s [ \E u_n(s_1) \E (u_n')^0(s-s_1)e^{i x S^0}  +  \E u_n'(s-s_1) \E u_n^0(x, s_1) e^{i x S^0}] \, ds_1\\
&+\frac{s}{n} \E[u_n'(s) - \E u_n'(s)] e^{i x S^0}.
\end{align*}
For the second term $\mathcal J_2$ we have
$$
\mathcal J_2 = Y_n(x, s).
$$
Let us consider now the term $\mathcal J_3$. We have
\begin{align*}
\mathcal J_3 &= \frac{ix}{\sqrt n} \sum_{j,k = 1}^n \E [\U(s) \HH^{(1)} ]_{k+n,j+n} \frac{\partial S}{\partial Y_{jk}^{(2)}} e^{i x S^0} \\
& = \frac{ix}{n} \sum_{j,k = 1}^n \E [\U(s) \HH^{(1)} ]_{k+n,j+n}  [\HH^{(2)} \J f'(\hat \V)]_{j+n, k+n} e^{i x S^0}  \\
& = \frac{x}{2n} \E \Tr \U'(s) f'(\hat \V) e^{i x S^0}.
\end{align*}
where we have applied the unitary matrix block decomposition~\eqref{eq: U matrix representation}
and used the following fact
$$
\int_{-\infty}^{\infty} u \hat f(u) \sum_{k=1}^n[\U_3(s) \W \HH (\Lambda(u) + \Lambda(-u)) \HH^{*}]_{kk} = 0
$$
which is valid for $f(\lambda)$ is an even function.

Finally we will have
\begin{align*}
I_{11} &=  -\frac{1}{n^2} \int_0^t \, ds \int_0^s \E u_n(s - s_1) \int_0^{s_1} [ \E u_n(s_2) \E (u_n')^0(s_1-s_2)e^{i x S^0}  \\
&\qquad\qquad\qquad\qquad\qquad+  \E u_n'(s_1-s_2) \E u_n^0 (x, s_2) e^{i x S^0} ]  \, ds_2 \, ds_1\\
&-\frac{1}{n^2} \int_0^t \, ds \int_0^s s_1 \E u_n(s - s_1) \E (u_n'(s_1))^0 \, ds_1 \\
&-\frac{1}{n} \int_0^t \, ds \int_0^s \E u_n(s_1) Y_n(x, s-s_1) \, ds_1\\
& - \frac{x}{2n^2}  \int_0^t \, ds \int_0^s \E u_n(s_1) \E \Tr \U'(s-s_1) f'(\hat \V) e^{i x S^0} \, ds_1.
\end{align*}
Changing the limits of integration, applying Lemma~\ref{l: trace variance} and $\E |u_n(t)| \le n$, we get
\begin{align*}
I_{11} &=  -\frac{2}{n^2} \int_0^t Y_n(x,s) \, ds \int_0^{t-s} \E u_n(s_1) \E u_n(t-s-s_1) \, ds_1  \\
& - \frac{x}{2n^2}  \int_0^t \E u_n(s) \E \Tr (\U(t-s) - \I) f'(\hat \V) e^{i x S^0} \, ds + r_n(t).
\end{align*}
Finally for the term $I_1$ we may write the following representation.
\begin{align*}
I_{1} &=  -\frac{3}{n^2} \int_0^t Y_n(x,s) \, ds \int_0^{t-s} \E u_n(s_1) \E u_n(t-s-s_1) \, ds_1  \\
& -\frac{x Z_n(x)}{2n^2}  \int_0^t \E u_n(s) \E \Tr (\U(t-s) - \I) f'(\hat \V) \, ds + r_n(t).
\end{align*}
It remains to calculate the term $I_2$.
\begin{align*}
I_2 &= -\frac{x}{\sqrt n} \int_0^t \sum_{j, k = 1}^n \E [\HH^{(2)} \J \U(s)]_{k, j } \frac{\partial S}{\partial Y_{jk}^{(1)} } e^{i x S^0} \, ds \\
& = -\frac{x}{n} \int_0^t \sum_{j, k = 1}^n \E [\HH^{(2)} \J \U(s)]_{k, j }  [f'(\hat \V) \HH^{(1)} ]_{j, k+n} e^{i x S^0} \, ds \\
& = -\frac{x Z_n(x)}{n} \int_0^t \sum_{k = 1}^n \E [\HH^{(2)} \J \U(s)f'(\hat \V) \HH^{(1)} ]_{k, k+n} \, ds
\end{align*}
where we have used the following observation. First of all we may write
\begin{align*}
&[\HH^{(2)} \J \U(s) \U(u) \HH^{(1)} ]_{k, k+n} = \sum_{j=1}^{2n}  [\HH^{(2)} \J \U(s)]_{k, j } [\U(u) \HH^{(1)} ]_{j, k+n} \\
&=\Tr [(\Y^{(2)})^T \Y^{(2)} \U_3(s) \U_2(u)] + \Tr [(\Y^{(2)})^T \Y^{(2)} \U_4(s) \U_4(u)] \\
& = \sum_{j = 1}^n \E [\HH^{(2)} \J \U(s)]_{k, j }  [\U(u) \HH^{(1)} ]_{j, k+n} + \Tr [(\Y^{(2)})^T \Y^{(2)} \U_4(s) \U_4(u)].
\end{align*}
From the representation~\eqref{eq: U matrix representation} it follows that
$$
\U_4(s) \U_4(u) = 4 \HH \D(s,u) \HH^{*},
$$
where $\D(s,u)$ is a diagonal matrix with $D_{jj}(s,u) = \cos (s_j s) \cos (s_j u), j = 1, ... , n$. Since $\hat f(t)$ is an even function we have
$$
\int_{-\infty}^\infty u \hat f(u) \int_0^t \E \Tr [(\Y^{(2)})^T \Y^{(2)} \HH \D(s,u) \HH^{*}] \, ds \, du = 0.
$$
We investigate now the behavior of
\begin{equation}\label{eq: term V}
\sum_{k = 1}^n [\HH^{(2)} \J \U(t)f'(\hat \V) \HH^{(1)} ]_{k, k+n}.
\end{equation}
Applying the same arguments as before
\begin{align*}
&\E \sum_{k = 1}^n [\HH^{(2)} \J \U(t)f'(\hat \V) \HH^{(1)} ]_{k, k+n} = \frac{1}{\sqrt n}\sum_{j, k = 1}^n \E Y_{jk}^{(2)} [\U(t)f'(\hat \V) \HH^{(1)} ]_{k+n, j+n}\\
&=\frac{1}{\sqrt n}\sum_{j, k = 1}^n \E \left [\frac{\partial \U(t)}{\partial Y_{jk}^{(2)} } f'(\hat \V) \HH^{(1)} \right ]_{k+n, j+n} +
\frac{1}{\sqrt n}\sum_{j, k = 1}^n \E \left [\U(t)f'(\hat \V) \frac{\partial \HH^{(1)}}{\partial Y_{jk}^{(2)} } \right ]_{k+n, j+n} \\
&=: T_1 + T_2.
\end{align*}
The term $T_1$ may be expanded in the sum of two terms
\begin{align*}
& T_1 = \frac{1}{\sqrt n}\sum_{j, k = 1}^n \sum_{l=1}^{2n} \E \left [\frac{\partial \U(t)}{\partial Y_{jk}^{(2)} } \right ]_{k+n,l} [f'(\hat \V) \HH^{(1)} ]_{l, j+n} \\
& =  \frac{i}{n}\int_0^t \sum_{k = 1}^n \E U_{k+n,k+n}(s) \sum_{j = 1}^n [\HH^{(2)} \J \U(t-s) f'(\hat \V) \HH^{(1)} ]_{j+n, j+n} \, ds\\
&+ \frac{i}{n}\int_0^t \sum_{k,j = 1}^n \E [\U(s) f'(\hat \V) \HH^{(1)}]_{k+n,j+n} [\HH^{(2)} \J \U(t-s)]_{j+n,k+n} \\
& =  \frac{1}{2n}\int_0^t \E u_n(s) \Tr \U'(t-s) f'(\hat \V) \, ds\\
&+ \frac{i}{n}\int_0^t \sum_{k,j = 1}^n \E [\U(s) f'(\hat \V) \HH^{(1)}]_{k+n,j+n} [\HH^{(2)} \J \U(t-s)]_{j+n,k+n}.
\end{align*}
For the term $T_2$ we get
$$
T_2 = \sum_{k=1}^n \E [\U(t) f'(\hat \V)]_{k+n, k+n} = \frac{1}{2}\E\Tr \U(t)f'(\hat \V).
$$
We get the following decomposition for~\eqref{eq: term V}
\begin{align*}
&\E \sum_{k = 1}^n [\HH^{(2)} \J \U(t)f'(\hat \V) \HH^{(1)} ]_{k, k+n} = \frac{1}{n}\int_0^t \E u_n(s)\Tr \U'(t-s) f'(\hat \V) \, ds \\
&\qquad\qquad\qquad+\frac{1}{2}\Tr \U(t)f'(\hat \V) \\
&\qquad\qquad\qquad+\frac{i}{n}\int_0^t \sum_{k,j = 1}^n \E [\U(s) f'(\hat \V) \HH^{(1)}]_{k+n,j+n} [\HH^{(2)} \J \U(t-s)]_{j+n,k+n}.
\end{align*}
Inserting this equation to $I_2$ we will have
\begin{align*}
&I_2  = -\frac{x Z_n(x)}{2 n^2} \int_0^t \int_0^s \E u_n(s_1) \Tr \U'(s-s_1) f'(\hat \V) \, ds_1 \, ds \\
&\qquad\qquad\qquad -\frac{x Z_n(x)}{2 n}\int_0^t  \Tr \U(s)f'(\hat \V) \, ds + r_n(t).
\end{align*}
Changing the limits of integration and  applying Lemma~\ref{l: trace variance}, we get
\begin{align*}
&I_2  = -\frac{x Z_n(x)}{2 n^2} \int_0^t \E u_n(s) \E \Tr [\U(t-s) f'(\hat \V)] \, ds \\
&\qquad\qquad\qquad\qquad-\frac{x Z_n(x)}{2 n}\int_0^t  \Tr \U(s)f'(\hat \V) \, ds + r_n(t),
\end{align*}
where we have also used that $f(\lambda)$ and $\E u(s)$ are even functions. It follows from~\eqref{eq: mathcal A representation} that we have derived representation for $\mathcal A_1$. The same arguments are valid for $\mathcal A_2$.

To simplify our notations let us introduce the following quantity
\begin{align*}
A_n(t): &=  - \frac{1}{2n}  \E \Tr [\U(t) f'(\hat \V)].
\end{align*}
One may see that $A_n(t)$ depends on $t$ only, but $Z_n(x)$ depends on $x$ only.
We derive an equation for $Y_n(x,t)$:
\begin{align}\label{eq: Y_n equation}
&Y_n(x,t) + 3 \int_0^t Y_n(x,s) v_n^{2*}(t-s)\, ds \nonumber\\
&\qquad\qquad\qquad\qquad = x Z_n(x) \int_0^t [v_n(s) A_n(t-s) + A_n(s)]\, ds + r_n(x,t),
\end{align}
where
$$
v_n(t) : = \frac{1}{n} \E u_n(t).
$$
As $n$ goes to infinity the sequence $v_n(t)$ uniformly converges to the following function
$$
v(t) = \int_{-a}^{a} e^{i t x} p(x) \, dx,
$$
where
\begin{align*}
p(x): = |x| P_2(x^2) \text{ and } a:= \sqrt K_2,
\end{align*}
with $P_2(x), K_2$ defined in Appendix~\ref{Fuss-Catalan Distribution}. This function is a characteristic function of Fuss-Catalan distribution.
The same arguments lead to
$$
A(t): = \lim_{n \rightarrow \infty} A_n(t) = - \int_{-a}^a e^{it\lambda} f'(\lambda) p(\lambda) \, d\lambda.
$$
Taking a limit in~\eqref{eq: Y_n equation} with respect to $n_l \rightarrow \infty$ we get
\begin{align}\label{eq: Y equation}
&Y(x,t) + 3 \int_0^t Y(x,s) v^{2*}(t-s)\, ds \nonumber \\
&\qquad\qquad\qquad\qquad\qquad= x Z(x) \int_0^t [2 v(s) A(t-s) + A(s)]\, ds,
\end{align}
Denote by $F(z), V(z)$ and $R(z)$ the generalized Fourier transform of $Y(x,t), v(t)$ and $A(t)$ respectively (see Appendix~\ref{ap: laplace transform}. Applying Statement~\ref{stat: Laplace transform}
we get from~\eqref{eq: Y equation}
$$
F(z) - 3 F(z) V^2(z) = - 2 i x Z(x) R(z) V(z) + \frac{i x Z(x) R(z)}{z}.
$$
and it follows that
\begin{equation}\label{eq: F(z) equation}
F(z) = \frac{ i x Z(x)  ( - 2 R(z) V(z) + R(z)/z ) }{1 - 3 V^2(z)}.
\end{equation}
It is easy to check that
$$
V(z) = s(z),
$$
where $s(z)$ is the Stieltjes transform of $p(x)$. In these notations we may rewrite~\eqref{eq: F(z) equation}
as follows
\begin{equation}\label{eq: F(z) equation 2}
F(z) = \frac{ i x Z(x)  ( - 2 R(z) s(z) + R(z)/z ) }{1 - 3 s^2(z)}.
\end{equation}
By Lemma~\ref{l: FT of K} the inverse Fourier transform of
$$
\frac{1/z - 2 s(z)}{1 - 3 s^2(z)}.
$$
is given by
\begin{equation}\label{eq: T(t) equation}
T(t) =  \frac{1}{\pi}\int_{-a}^a \frac{ e^{it\mu}}{3p_1(\mu)}  \frac{4p_1(\mu)^4 + 11 p_1(\mu)^2 + 4}{4 p_1(\mu)^2 + 3} \, d\mu,
\end{equation}
where $p_1(\mu) := \pi p(\lambda)$.
From~\eqref{eq: F(z) equation 2} and~\eqref{eq: T(t) equation} we conclude
\begin{align*}
&Y(x,t) = - \frac{x Z(x)}{\pi^2} \int_0^t \int_{-a}^a e^{is \lambda} f'(\lambda) p(\lambda) \, d\lambda \\
&\qquad\qquad\qquad\times\int_{-a}^{a} \frac{ e^{i(t-s)\mu}}{3p(\mu)} \frac{4p_1(\mu)^4 + 11 p_1(\mu)^2 + 4}{4 p_1(\mu)^2 + 3} \, dv
\end{align*}
Simple calculation yields
\begin{align}\label{eq: Y equation 2}
&Y(x,t) =   \frac{i x Z(x)}{\pi^2}  \int_{-a}^a p(\lambda) \, d\lambda \nonumber\\
&\qquad\qquad\qquad\times\int_{-a}^a \frac{e^{i t \lambda} - e^{i t \mu}}{\lambda-\mu}
f'(\lambda)  \frac{1}{3p(\mu)} \frac{4p_1(\mu)^4 + 11 p_1(\mu)^2 + 4}{4 p_1(\mu)^2 + 3}  \, d\mu.
\end{align}
Finally we get from~\eqref{eq: Z_n derivative equation} and~\eqref{eq: Y equation 2}
\begin{align}\label{eq: Z_n derivative equation 2}
&\lim_{n_l \rightarrow \infty} Z_n'(x,t) =   - \frac{ x Z(x)}{\pi^2}  \int_{-a}^a p(\lambda) \, d\lambda \nonumber\\
&\qquad\qquad\times\int_{-a}^a \frac{f(\lambda) - f(\mu)}{\lambda-\mu}
f'(\lambda)  \frac{1}{3p(\mu)} \frac{4p_1(\mu)^4 + 11 p_1(\mu)^2 + 4}{4 p_1(\mu)^2 + 3}  \, d\mu.
\end{align}
One may see that
$$
(f(\lambda) - f(\mu))f'(\lambda) = \frac{1}{2}\frac{\partial}{\partial \lambda}(f(\lambda) - f(\mu))^2,
$$
and~\eqref{eq: Z_n derivative equation 2} may be rewritten applying integration by parts in the following way
\begin{align*}
&\lim_{n_l \rightarrow \infty} Z_n' =   - \frac{ x Z(x)}{2\pi^2}  \int_{-a}^a \int_{-a}^a \frac{(f(\lambda) - f(\mu))^2}{(\lambda-\mu)^2} \\
&\qquad\qquad\qquad\times \frac{[p(\lambda)-p'(\lambda)(\lambda-\mu)]}{3p(\mu)} \frac{[4p_1(\mu)^4 + 11 p_1(\mu)^2 + 4]}{4 p_1(\mu)^2 + 3}  \, d\lambda \, d\mu.
\end{align*}
Comparing this with~\eqref{eq: we have to show} we may conclude the proof of Theorem~\ref{th:main}.
\end{proof}

\section{The General Case} \label{section: general case}
In this section we finish the proof of Theorem~\ref{th:main general case}. Applying the method of Bentkus from~\cite{Bentkus2003} and the method of Tikhomirov from~\cite{Tikh1980},\cite{Tikh2001} we show that one may substitute the general matrix by the matrix with i.i.d. Gaussian random variables and express the characteristic function in the general case via the characteristic function in the Gaussian case.  These methods have been applied several times in random matrix theory, see, for example,~\cite{GotNauTikh2012} and~\cite{GotzNauTikh2014} .

\subsection{Truncation}
In this subsection we show by standard arguments that we may truncate the entries of $\X^{(q)}, q = 1,2$. For all $1 \le j,k\le n, q = 1,2$, we introduce truncated random variables $\hat X_{jk}^{(q)}: = X_{jk}^{(q)} \mathbb I(|X_{jk}^{(q)}| \geq \tau \sqrt n)$. Denote
by $\hat \X^{(q)}: = [\hat X_{jk}^{(q)}]_{j,k=1}^n$. One may see that
\begin{align}\label{eq: truncation}
\Pb (\hat \X^{(q)} \neq \X^{(q)}) \le \sum_{j,k = 1}^n \E \mathbb I(|X_{jk}^{(q)}| \geq \tau \sqrt n) \le \frac{1}{\tau^4}\E X_{11}^4 \mathbb I(|X_{jk}^{(q)}| \geq \tau \sqrt n).
\end{align}
Let $\hat S^0$ denote $S^0$ with all $X_{jk}^{(q)}$ replaced by $\hat X_{jk}^{(q)}$. It follows from~\eqref{eq: truncation} that
$$
\lim_{n \rightarrow \infty}|\E e^{itS^0} - \E e^{i t \hat S^0} |  = 0.
$$
By the similar arguments one may show that we may assume that $\E \hat X_{jk}^{(q)} = 0$ and $\Var(\hat X_{jk}^{(q)}) = 1$.
In what follows we will assume that $|X_{jk}^{(q)}| \le \tau \sqrt n$.
\begin{remark}
It is easy to see that one may assume $X_{jk}^{(q)}$ non i.i.d., but imply the following Lendeberg type condition on the fourth moments
$$
\text{ for all } \tau > 0 \quad \lim_{n \rightarrow \infty} \frac{1}{n^2} \sum_{j,k=1}^n \E (X_{jk}^{(q)})^4\mathbb I(|X_{jk}^{(q)}| \geq \tau \sqrt n)  = 0.
$$
\end{remark}
\subsection{From the general case to the Gaussian case}

Let $\Y^{(1)},\Y^{(2)}$ be $n\times n$ independent random matrices  with independent Gaussian entries $n^{-1/2} Y_{jk}^{(q)}$ such that
\begin{align*}
&\E Y_{jk}^{(q)}=0,\qquad \E (Y_{jk}^{(q)})^2=1,\qquad\text{for any}\quad q=1,2,\, j,k=1\ldots,n.
\end{align*}
For any $\varphi\in[0,\frac{\pi}2]$ and any $\nu=1, 2$, introduce the matrices
\begin{equation*}
\Z^{(q)}(\varphi)=\X^{(q)}\sin\varphi+\Y^{(q)}\cos\varphi
\end{equation*}
where
$$[\Z^{(q)}(\varphi)]_{jk}=\frac1{\sqrt{n}}Z_{jk}^{(q)}= \frac1{\sqrt{n}}(X_{jk}^{(q)}\sin\varphi+Y_{jk}^{(q}\cos\varphi).
$$
We define the matrices $\W(\varphi)$, $\HH^{(q)}(\varphi)$, $\V(\varphi)$, $\widehat{\V}(\varphi)$, $\U(\varphi, t)$ by
\begin{align*}
\W(\varphi)=\prod_{q=1}^2\Z^{(q)}(\varphi),\quad
\HH^{(q)}(\varphi)=\begin{bmatrix}&\Z^{(q)}(\varphi))&\OO\\&\OO&\Z^{(m-q+1)}
(\varphi)\end{bmatrix}\\
\V(\varphi)=\prod_{q=1}^2\HH^{(q)}(\varphi),\quad
\widehat{\V}(\varphi)=\V(\varphi)\J,\quad
\U(\varphi, t )= e^{- i t \hat \V}
\end{align*}
Recall that $\I$ (with sub-index or without it) denotes the unit matrix of corresponding order, $\J=\begin{bmatrix}&\OO&
\I\\ &\I&\OO\end{bmatrix}$ and
$\OO$ denotes the matrix with zero-entries.

Let $S(\varphi) = \Tr f(\Z(\varphi)), S^0(\varphi) = S^0(\varphi) - \E S^0(\varphi)$ and
$$
g_n(t,\varphi) = \E e^{i t S^0(\varphi)}
$$

\begin{theorem}\label{universality}
Under \Cond \, $\lim_{n \rightarrow \infty} g_n(t)$ satisfies the following equation
$$
g(t, \pi/2) = g(t,0) e^{-t^2 \kappa_4 \Psi^2/4},
$$
where
\begin{equation}\label{eq: varphi square}
\Psi =  \int_{-a}^a f(\lambda) [p(\lambda) + \lambda p'(\lambda)] \, d\lambda
\end{equation}
and $\kappa_4 = \mu_4 - 3$.
\end{theorem}
\begin{proof}
We prove that the function $\lim_{n \rightarrow \infty} g_n(t,\varphi)$ satisfies the following equation
$$
\frac{\partial g(t, \varphi)}{\partial \varphi} = - \kappa_4 t^2 \sin^3 \varphi \cos \varphi \Psi^2 g(t,\varphi).
$$
It will follow from this equation that
$$
g(t, \pi/2) = g(t,0) \exp\left \{-\kappa_4 t^2 \Psi^2 \int_0^{\pi/2} \sin^3 \alpha \cos \alpha \, d\alpha \right \}
$$
Note that
$$
g_n(t, \pi/2) - g_n(t,0) = \int_0^{\pi/2} \frac{\partial g_n(t, \varphi)}{\partial \varphi} \, d\varphi
$$
Similarly to the section~\ref{section: gaussian case} it is sufficient to prove that any converging subsequences $\{g_{n_l}\}$ and $\left\{\frac{\partial g_{n_l}}{\partial \varphi}\right\}$ satisfy
\begin{equation*}
\lim_{n_l \rightarrow \infty} g_{n_l}(t, \varphi) = g(t, \varphi), \quad \lim_{n_l \rightarrow \infty} \frac{\partial g_{n_l}(t,\varphi)}{\partial \varphi} = - \kappa_4 t^2 \sin^3 \varphi \cos \varphi \Psi^2 g(t,\varphi).
\end{equation*}
By Lemma~\ref{l: S derivative} we get
\begin{equation}\label{eq: dg devidev by dphi 1}
\frac{\partial g_n(t, \varphi)}{\partial \varphi} = \frac{i t}{\sqrt n} \sum_{q=1}^2 \sum_{j,k=1}^n \E \hat Z_{jk}^{(q)} [\V_{[m-q+2,m]}\J f'(\hat \V) \V_{[1, m-q]}]_{j+n,k+n} e^{i t S^0},
\end{equation}
where
$$
\hat Z_{jk}^{(q)} := \frac{d}{d \varphi} Z_{jk}^{(q)} = X_{jk}^{(q)} \cos \varphi - Y_{jk}^{(q)} \sin \varphi.
$$
It is straightforward to check that
\begin{align}\label{eq:relation between moments}
&\E \hat Z_{jk}^{(q)} (Z_{jk}^{(q)})^p = 0, \text{ for } p = 0, 1; \\
&\E \hat Z_{jk}^{(q)} (Z_{jk}^{(q)})^2 = \E (X_{jk}^{(q)})^3 \cos^3 \varphi;\nonumber\\
\label{eq:relation between moments 2}
&\E \hat Z_{jk}^{(q)} (Z_{jk}^{(q)})^3 = \kappa_4 \sin^3 \varphi \cos \varphi.
\end{align}
Let us introduce further notations. Denote by $\V_{[\alpha, \beta]}^{j,k,q}$(x) the corresponding matrix $\V_{[\alpha, \beta]}$ with $Z_{jk}^{(q)}$ replaced by $x$. Let us also denote
$$
\Phi_{jkq}(x) := [\V_{[m-q+2,m]}^{(j,k, q)}(x)\J f'(\hat \V^{(j,k, q)}(y)) \V_{[1, m-q]}^{(j,k, q)}(x)]_{j+n,k+n} e^{i t S^0(\V^{(j,k, q)}(x))}.
$$
Applying Taylor's formula we get
\begin{align*}
\Phi_{jkq} (Z_{jk}^{(q)}) &= \sum_{p=0}^3 \frac{1}{p!}((Z_{jk}^{(q)})^p \Phi_{jkq}^{(p)}(0) +\frac{1}{3!}(Z_{jk}^{(q)})^4 \E_\theta (1 - \theta)^3 \Phi_{jkq}^{(4)}(\theta Z_{jk}^{(q)})
\end{align*}
This equation and~\eqref{eq: dg devidev by dphi 1} together imply
\begin{align*}
\frac{\partial g_n(t, \varphi)}{\partial \varphi} &= \frac{it}{n^{1/2}}\sum_{p=1}^3 \frac{1}{p!} \sum_{q=1}^2 \sum_{j,k=1}^n  \E \hat Z_{jk}^{(q)} (Z_{jk}^{(q)})^p  \E \Phi_{jkq}^{(p)} (0) \\
&+\frac{it}{3!n^{1/2}}\sum_{q=1}^2\sum_{j,k=1}^n \E (1 - \theta)^3 \hat Z_{jk}^{(q)} (Z_{jk}^{(q)})^4 \Phi_{jkq}^{(4)}(\theta Z_{jk}^{(q)})\\
& = T_1 + ... +T_4.
\end{align*}
It follows from~\eqref{eq:relation between moments} that $T_1 = 0$. In the next subsections we will investigate the term $T_k, k = 2, 3, 4$.

\subsection{The second derivative}\label{subs: second derivative}
First we note that
$$
[\V_{[m-q+2,m]}^{(j,k, q)}\J \A \V_{[1, m-q]}^{(j,k, q)}]_{j+n,k+n} = [\V_{[m-q+2,m]}\J \A \V_{[1, m-q]}]_{j+n,k+n}.
$$
for an arbitrary matrix $\A$.
It is straightforward to check that
$$
\Phi_{jkq}^{(2)}(0) = L_{jkq}^1 + L_{jkq}^2 + L_{jkq}^3,
$$
where
\begin{align*}
&L_{jkq}^1 = [\V_{[m-q+2,m]}\J \frac{\partial^2 f'(\hat \V)}{\partial (Z_{jk}^{(q)})^2}\bigg |_{Z_{jk}^{(q)}=0} \V_{[1, m-q]}]_{j+n,k+n} e^{i t S^0}, \\
&L_{jkq}^2 =  \frac{3 i t}{\sqrt n} [\V_{[m-q+2,m]} \J \frac{\partial f'(\hat \V)}{\partial Z_{jk}^{(q)}}\bigg |_{Z_{jk}^{(q)}=0} \V_{[1, m-q]}]_{j+n,k+n}  \\
&\qquad\qquad\qquad\qquad\qquad\times[\V_{[m-q+2,m]}\J f'(\hat \V)\big |_{Z_{jk}^{(q)}=0} \V_{[1, m-q]}]_{j+n,k+n}   e^{i t S^0}, \\
&L_{jkq}^3 = - \frac{t^2}{n} [\V_{[m-q+2,m]} \J f'(\hat \V)\big |_{Z_{jk}^{(q)}=0} \V_{[1, m-q]}]_{j+n,k+n}^3   e^{i t S^0}.
\end{align*}
Let us consider, for example, the term $\frac{1}{\sqrt n} L_{jkq}^2$. We have, for $q=1$,
\begin{align*}
\frac{1}{n^{1/2}} \sum_{j,k=1}^n \E L_{jkq}^2 = I_1 + I_2,
\end{align*}
where
\begin{align*}
&I_1 = \frac{3t}{n^{3/2}} \int_{-\infty}^\infty \int_{-\infty}^\infty u v \hat f(u) \hat f(v) \int_0^u
\E \sum_{j,k=1}^n [\U(s) \HH^{(1)}]_{j,k+n} \\
&\qquad\qquad\qquad\qquad \times [\U(u-s) \HH^{(1)}]_{j,k+n} [\U(v) \HH^{(1)}]_{j,k+n} \, ds \, du \, dv
\end{align*}
and
\begin{align*}
&I_2 = \frac{3t}{n^{3/2}} \int_{-\infty}^\infty \int_{-\infty}^\infty u v \hat f(u) \hat f(v)  \int_0^u
\E \sum_{j,k=1}^n [(\HH^{(1)})^T \U(s) \HH^{(1)}]_{k+n,k+n} \\
&\qquad\qquad\qquad\qquad \times [\U(u-s)]_{j,j} [\U(v) \HH^{(1)}]_{j,k+n} \, ds \, du \, dv
\end{align*}
We estimate the term $I_1$. Applying Cauchy-Schwarz inequality and orthogonality relations for $\U$ we get
\begin{align*}
&\frac{1}{n^{3/2}}\E \left |\sum_{j,k=1}^n [\U(s) \HH^{(1)}]_{j,k+n}[\U(u-s) \HH^{(1)}]_{j,k+n} [\U(v) \HH^{(1)}]_{j,k+n} \right|\\
&\le \frac{1}{n^{3}}\E \left [\sum_{l_1, l_2, l_3 =1}^{n} \left (\sum_{k=1}^n Z_{k l_1}^{(2)} Z_{k l_2}^{(2)} Z_{k l_3}^{(2)} \right)
\left (\sum_{j=1}^n [\U_2]_{jl_1}(s) [\U_2]_{jl_2}(u- s)[\U_2]_{jl_3}(v) \right )\right ]\\
&\le \frac{1}{n^{5/2}}\E^{1/2} \left [\sum_{l_1, l_2, l_3 =1}^n \left (\sum_{k=1}^n Z_{k l_1}^{(2)} Z_{k l_2}^{(2)} Z_{k l_3}^{(2)}\right)^2\right]\\
&\le \frac{1}{n^{5/2}}\E^{1/2} \left[ \sum_{l=1}^n (\sum_{k=1}^n (Z_{k l}^{(2)})^3)^2 + 3 \sum_{l_1 \neq l_2} (\sum_{k=1}^n (Z_{k l_1}^{(2)})^2 Z_{k l_2}^{(2)})^2 \right.\\
&\qquad\qquad\qquad\left.+\sum_{l_1\neq l_2 \neq l_3} \sum_{k=1}^n (Z_{k l_1}^{(2)})^2 (Z_{k l_2}^{(2)})^2 (Z_{k l_3}^{(2)})^2 \right] \\
&\le C (\tau + 1) n^{-1/2}.
\end{align*}
For the term $I_2$ we may apply the same arguments. Finally
$$
|I_1 + I_2| \le \frac{C t (\tau + 1)}{n^{1/2}} \int_{\infty}^\infty |u|^3 |\hat f(u)| \, du \int_{\infty}^\infty |v| |\hat f(v)| \, dv
$$
Analogously one may show that $\frac{1}{n^{1/2}}\sum_{j,k=1}^n L_{jkq}^1$ and $\frac{1}{n^{1/2}}\sum_{j,k=1}^n L_{jkq}^3$
goes to zero as $n$ goes to infinity. It follows that $T_2 = o(1)$.

\subsection{The third derivative}\label{subs: third derivative}
We investigate now the term $T_3$. Direct computations yield
\begin{align*}
\Phi_{jkq}^{(3)}(0) &= [\V_{[m-q+2,m]} \J \frac{\partial^3 f'(\hat \V)}{\partial (Z_{jk}^{(q)})^3}\bigg |_{Z_{jk}^{(q)}=0} \V_{[1, m-q]}]_{j+n,k+n} e^{i t S^0} \\
&+\frac{4 i t}{\sqrt n} [\V_{[m-q+2,m]}\J \frac{\partial^2 f'(\hat \V)}{\partial (Z_{jk}^{(q)})^2}\bigg |_{Z_{jk}^{(q)}=0} \V_{[1, m-q]}]_{j+n,k+n} \\
&\times[\V_{[m-q+2,m]}\J f'(\hat \V)\bigg |_{Z_{jk}^{(q)}=0} \V_{[1, m-q]}]_{j+n,k+n}   e^{i t S^0} \\
&+\frac{3 i t}{\sqrt n} [\V_{[m-q+2,m]} \J \frac{\partial f'(\hat \V)}{\partial Z_{jk}^{(q)}}\bigg |_{Z_{jk}^{(q)}=0} \V_{[1, m-q]}]_{j+n,k+n}^2 e^{i t S^0} \\
&-\frac{6t^2}{n} [\V_{[m-q+2,m]} \J f'(\hat \V)\bigg |_{Z_{jk}^{(q)}=0} \V_{[1, m-q]}]_{j+n,k+n}^2\\
&\times [\V_{[m-q+2,m]} \J \frac{\partial f'(\hat \V)}{\partial Z_{jk}^{(q)}}\bigg |_{Z_{jk}^{(q)}=0} \V_{[1, m-q]}]_{j+n,k+n} e^{i t S^0}\\
& - \frac{i t^3}{n^{3/2}} [\V_{[m-q+2,m]}\J f'(\hat \V)\bigg |_{Z_{jk}^{(q)}=0} \V_{[1, m-q]}]_{j+n,k+n}^4   e^{i t S^0}
\end{align*}
It is straightforward to check that all terms except the third are of order $o(1)$. These may be done similarly to the previous section.
Let us denote
$$
\Psi_n^{q} = \frac{1}{n}\sum_{j,k=1}^n \E[\V_{[m-q+2,m]}\J \frac{\partial f'(\hat \V)}{\partial Z_{jk}^{(q)}}\bigg |_{Z_{jk}^{(q)}=0} \V_{[1, m-q]}]_{j+n,k+n}^2 e^{i t S^0}.
$$
Our aim is to show that
\begin{equation}\label{eq: limit of Psi n}
\lim_{n_l \rightarrow \infty} \Psi_{n_l}^q = \left[\int_{-a}^a f(\lambda) [p(\lambda) + \lambda p'(\lambda)] d\lambda\right]^2.
\end{equation}
Consider the case $q=1$. We get by Lemma~\ref{l: U derivative}
\begin{align*}
&\left[\frac{\partial f'(\hat \V)}{\partial Z_{jk}^{(1)}} \HH^{(1)}\right]_{j,k+n} = \sum_{l=1}^{2n} \left [\frac{\partial f'(\hat \V)}{\partial Z_{jk}^{(1)}}\right]_{j,l}  [\HH^{(1)}]_{l,k+n} \\
&=-\frac{1}{\sqrt n} \int_{-\infty}^{\infty} s \hat f(s) \sum_{l=1}^{2n} [\U \HH^{(1)}]_{j,k+n}*[\U]_{l, j} (s) [\HH^{(1)}]_{l,k+n} \\
&\qquad\qquad\qquad\qquad + \sum_{l=1}^{2n} [\U \HH^{(1)}]_{l,k+n}*[\U]_{j, j} (s) [\HH^{(1)}]_{l,k+n} \, ds\\
&=-\frac{1}{\sqrt n} \int_{-\infty}^{\infty} s \hat f(s) [\U \HH^{(1)}]_{j,k+n}* [\U \HH^{(1)}]_{j,k+n}(s) \\
&\qquad\qquad\qquad\qquad+ [(\HH^{(1)})^T \U \HH^{(1)}]_{k+n,k+n}*[\U]_{j, j} (s) \, ds
\end{align*}
Similarly to the previous sections we may show that the first term in the last equation has the zero impact.
It is straightforward to check
$$
[(\HH^{(1)})^T \U(s) \HH^{(1)}]_{k+n,k+n} = [\HH^{(2)} \J \U(s) \HH^{(1)}]_{k,k+n} = : T_k(s).
$$
Let us investigate the following integral
$$
\int_{-\infty}^{\infty} s \hat f(s) \E T_k*\E[\U]_{j, j} (s) \, ds
$$
We have
\begin{align*}
\E T_k(s) &= \frac{1}{\sqrt n}\sum_{l=1}^n \E Z_{kl}^{(2)} [\U(s) \HH^{(1)}]_{l+n,k+n} \\
&= \frac{1}{\sqrt n}\sum_{l=1}^n\E \left[\frac{\partial \U(s)}{\partial Z_{kl}^{(2)}} \HH^{(1)}\right]_{l+n,k+n} + \frac{1}{\sqrt n}\sum_{l=1}^n\E\left[\U(s) \frac{\partial \HH^{(1)}(s)}{\partial Z_{kl}^{(2)}} \right]_{l+n,k+n}.
\end{align*}
It follows that
\begin{align*}
&\sum_{l=1}^n \E\left[\frac{\partial \U(s)}{\partial Z_{kl}^{(2)}} \HH^{(1)}\right]_{l+n,k+n} =
\sum_{l=1}^n\sum_{p=1}^{2n} \E\left[\frac{\partial \U(s)}{\partial Z_{kl}^{(2)}}\right]_{l+n,p} [\HH^{(1)}]_{p,k+n}\\
&= \frac{i}{\sqrt n} \sum_{l=1}^n\sum_{p=1}^{2n} \E[\U]_{l+n,l+n}*[\HH^{(2)}\J \U]_{k+n, p} [\HH^{(1)}]_{p,k+n}\\
&+ \frac{i}{\sqrt n} \sum_{l=1}^n\sum_{p=1}^{2n} \E[\U ]_{p,l+n}*[\HH^{(2)} \J \U]_{k+n, l+n} [\HH^{(1)}]_{p,k+n}\\
&= \frac{i}{\sqrt n} \sum_{l=1}^n \E[\U]_{l+n,l+n}*[\HH^{(2)}\J \U \HH^{(1)}]_{k+n,k+n}\\
&+ \frac{i}{\sqrt n} \sum_{l=1}^n \E[\HH^{(2)} \J \U]_{k+n, l+n}* [\U \HH^{(1)}]_{l+n,k+n}.
\end{align*}
We may show that the second term has the zero impact. It follows that
\begin{equation}\label{eq: equation for T_k(t)}
\E T_k(s) = \frac{i}{n}\E u_n*\E[\HH^{(2)}\J \U \HH^{(1)}]_{k+n,k+n}(s) + \frac{1}{n}\E u_n(s) + o(1),
\end{equation}
where we have also applied Lemma~\ref{l: variance of trace general case}.
Finally we will have for $q = 1$
\begin{align}\label{eq: Psi_n^q}
&\Psi_n^q=\frac{g_n(\varphi,t)}{n^2}\sum_{j,k=1}^n \left\{\int_{-\infty}^{\infty} s \hat f(s) \left [\frac{1}{n}\E u_n*\E[\U]_{j, j}(s)\right.\right.\\
&\qquad\qquad\left.\left.+ \frac{i}{n}\E u_n *\E[\HH^{(2)}\J \U \HH^{(1)}]_{k+n,k+n}*\E[\U]_{j, j}(s)\right] \, ds \right\}^2 + o(1). \nonumber
\end{align}
Let us introduce further notations and denote
$$
V_{n,k}(s) := \E [\HH^{(2)} \J \U(s) \HH^{(1)}]_{k+n,k+n}.
$$
We may write, applying Lemma~\ref{l: U derivative}
\begin{align*}
&V_{n,k} = \frac{1}{\sqrt n}\sum_{j=1}^n \E Z_{jk}^{(1)} [\U \HH^{(1)}]_{j,k+n} = \sum_{j=1}^n \E \left [ \frac{\partial \U(s)}{\partial Z_{jk}^{(1)}}\right]_{j,k+n}\\
&=\frac{i}{n} \sum_{j=1}^n \sum_{l=1}^{2n} \E[\U \HH^{(1)}]_{j,k+n}*[\U]_{l,j}(s) [\HH^{(1)}]_{l,k+n} \\
&+\frac{i}{n} \sum_{j=1}^n \sum_{l=1}^{2n} \E[\U \HH^{(1)}]_{l,k+n}*[\U]_{j,j}(s) [\HH^{(1)}]_{l,k+n}
\end{align*}
The same arguments as before yield that
\begin{equation}\label{eq: V n k equation}
V_{n,k}(s) = \frac{i}{n}\E u_n(s)*\E T_k(s) + o(1).
\end{equation}
Applying~\eqref{eq: equation for T_k(t)} we get
\begin{equation*}
V_{n,k}(s) = \frac{i}{n^2}\E u_n*\E u_n(s) - \frac{1}{n^2}\E u_n * \E u_n * V_{n,k}(s) + o(1).
\end{equation*}
This means that $\lim_{n \rightarrow \infty} V_{n,k}$ satisfies the following equation
$$
h(s) = i u*u(s) - u*u*h(s)
$$
The same equation holds for $\lim_{n \rightarrow \infty} \frac{1}{n}\sum_{k=1}^n V_{n,k}(s)$.
Since
$$
\lim_{n \rightarrow \infty} \frac{1}{n}\sum_{k=1}^n V_{n,k}(s) = -i v'(s)
$$
that means that
$$
\lim_{n \rightarrow \infty} V_{n,k} = - i v'(s)
$$
Taking the limit with respect to $n_l \rightarrow \infty$ we get in~\eqref{eq: Psi_n^q}
\begin{align}\label{eq: limit Psi_n^q}
\lim_{n_l \rightarrow \infty} \Psi_n^q = g(\varphi,t)\left\{\int_{-\infty}^{\infty} s \hat f(s)  [ v*v(s)+ v*v*v'(s)] \, ds \right\}^2
\end{align}
Let us consider the following integral
$$
\int_{-\infty}^\infty s \hat f(s) [v*v(s) + v*v*v'(s)] \, ds
$$
First we investigate the Fourier transform of $v*v(s) + v*v*v'(s)$. It is given by
$$
i s^2(z) - i (1 + z s(z))s^2(z) = - i z s^3(z)
$$
By Proposition~\ref{stat: Laplace transform} it follows that
\begin{equation}\label{eq: inverse Fourier transform 1 general case}
v*v(t) + v*v*v'(s) = \frac{1}{2\pi} \int_L e^{i s z}z s^3(z) \, dz.
\end{equation}
Since $1 + z s(z) = z s^3(z)$, the right hand side of~\eqref{eq: inverse Fourier transform 1 general case} may be rewitten as
$$
\frac{1}{2\pi} \int_L e^{i s z}(1 + z s(z)) \, dz.
$$
Similarly to the proof of Lemma~\ref{l: FT of K} we get
$$
\frac{1}{2\pi} \int_L e^{i s z}(1 + z s(z)) \, dz = i\int_{-a}^a e^{i s \lambda} \lambda p(\lambda) \, d\lambda.
$$
Integrating by parts we will have
$$
i\int_{-a}^a e^{i s \lambda} \lambda p(\lambda) \, d\lambda = -\frac{1}{s} \int_{-a}^a e^{i s \lambda}[p(\lambda) + \lambda p'(\lambda)] \, d\lambda.
$$
Finally we conclude that
$$
\int_{-\infty}^\infty t \hat f(t)[v*v(s) + v*v*v'(s)] \, ds = \int_{-a}^a f(\lambda)[p(\lambda) + \lambda p'(\lambda)] \, d\lambda.
$$
and finish the proof of~\eqref{eq: limit of Psi n}.
If we show that for all $1 \le j,k \le n$
\begin{equation}\label{eq: the last argument}
\lim_{n \rightarrow \infty} \E \left [ \int_{-\infty}^\infty s \hat f(s) [T_k*[\U]_{jj}(s)- \E T_k * \E [\U]_{jj}(s)] \, ds \right ]^2 e^{itS^0} = 0.
\end{equation}
then from~\eqref{eq:relation between moments 2} and~\eqref{eq: limit of Psi n} we will have
$$
\lim_{n_l \rightarrow \infty}  \frac{it}{3! n_l^{1/2}} \sum_{q=1}^2 \sum_{j,k=1}^{n_l}  \E \hat Z_{jk}^{(q)} (Z_{jk}^{(q)})^3  \E \Phi_{jkq}^{(p)} (0)  =
-\kappa_4 t^2 \sin^3 \varphi \cos \varphi \Psi^2 g(t, \varphi),
$$
where $\Psi$ is given by~\eqref{eq: varphi square}.

To prove~\eqref{eq: the last argument} it is enough to show that
\begin{align}\label{eq: var of T_k}
\Var[T_k(s)] = o(1), \quad k = 1, ... , n; \\
\label{eq: var of U_jj}
\Var([\U]_{jj}) = o(1), \quad j = 1, ... , n.
\end{align}
We may apply Lemma~\ref{l: variance of trace general case} to conclude the proof of Theorem.

\subsection{The remainder term}
To conclude the proof of Theorem~\ref{th:main general case} it remains to estimate the remainder term $T_4$. One may see that $\E \hat Z_{jk}^{(q)} (Z_{jk}^{(q)})^4 \le C \tau \sqrt n \mu_4$.
Let $Z$ be a random variable which has the same distribution as $Z_{11}^{(1)}$. We estimate $\E \sup_Z \Phi_{jkq}^{(3)}(Z)$. Simple calculations yield that
$$
\Phi_{jkq}^{(3)}(Z) = L_{jkq}^1 + ... + L_{jkq}^7,
$$
where
\begin{align*}
&L_{jkq}^1 = \frac{1}{n^2}[\V_{[m-q+2,m]}\J \frac{\partial^4 f'(\hat \V)}{\partial (Z_{jk}^{(q)})^4}\bigg |_{Z_{jk}^{(q)}=Z} \V_{[1, m-q]}]_{j+n,k+n} e^{i t S^0}, \\
&L_{jkq}^2 = \frac{5it}{\sqrt n}[\V_{[m-q+2,m]}\J \frac{\partial^3 f'(\hat \V)}{\partial (Z_{jk}^{(q)})^3}\bigg |_{Z_{jk}^{(q)}=Z} \V_{[1, m-q]}]_{j+n,k+n} \\
&\qquad\qquad\qquad\qquad\times[\V_{[m-q+2,m]} \J f'(\hat \V)\big |_{Z_{jk}^{(q)}=Z} \V_{[1, m-q]}]_{j+n,k+n} e^{i t S^0},\\
&L_{jkq}^3 = \frac{10it}{\sqrt n}[\V_{[m-q+2,m]}\J \frac{\partial^2 f'(\hat \V)}{\partial (Z_{jk}^{(q)})^2}\bigg |_{Z_{jk}^{(q)}=Z} \V_{[1, m-q]}]_{j+n,k+n} \\
&\qquad\qquad\qquad\qquad\times[\V_{[m-q+2,m]} \J \frac{\partial f'(\hat \V)}{\partial Z_{jk}^{(q)}}\big |_{Z_{jk}^{(q)}=Z} \V_{[1, m-q]}]_{j+n,k+n} e^{i t S^0}, \\
&L_{jkq}^4 = -\frac{10t^2}{n}[\V_{[m-q+2,m]}\J \frac{\partial^2 f'(\hat \V)}{\partial (Z_{jk}^{(q)})^2}\bigg |_{Z_{jk}^{(q)}=Z} \V_{[1, m-q]}]_{j+n,k+n} \\
&\qquad\qquad\qquad\qquad\times[\V_{[m-q+2,m]} \J f'(\hat \V)\big |_{Z_{jk}^{(q)}=Z} \V_{[1, m-q]}]_{j+n,k+n}^2 e^{i t S^0},\\
&L_{jkq}^5 = -\frac{15t^2}{n}[\V_{[m-q+2,m]}\J \frac{\partial f'(\hat \V)}{\partial Z_{jk}^{(q)}}\bigg |_{Z_{jk}^{(q)}=Z} \V_{[1, m-q]}]_{j+n,k+n}^2  \\
&\qquad\qquad\qquad\qquad\times[\V_{[m-q+2,m]} \J f'(\hat \V)\big |_{Z_{jk}^{(q)}=Z} \V_{[1, m-q]}]_{j+n,k+n} e^{i t S^0},\\
&L_{jkq}^6 = -\frac{10it^3}{n^{3/2}}[\V_{[m-q+2,m]}\J \frac{\partial f'(\hat \V)}{\partial Z_{jk}^{(q)}}\bigg |_{Z_{jk}^{(q)}=Z} \V_{[1, m-q]}]_{j+n,k+n} \\
&\qquad\qquad\qquad\qquad\times[\V_{[m-q+2,m]} \J f'(\hat \V)\big |_{Z_{jk}^{(q)}=Z} \V_{[1, m-q]}]_{j+n,k+n}^3 e^{i t S^0}, \\
&L_{jkq}^7 = \frac{t^4}{n^2}[\V_{[m-q+2,m]} \J f'(\hat \V)\big |_{Z_{jk}^{(q)}=Z} \V_{[1, m-q]}]_{j+n,k+n}^5 e^{i t S^0}.
\end{align*}
Applying the same arguments as before in subsections~\ref{subs: second derivative} and~\ref{subs: third derivative} we get that
$$
|T_4| \le C \tau.
$$
One may show that it is possible to change $\tau$ by the sequence $\tau_n$, such that $\lim_{n \rightarrow \infty} \tau_n = 0$ and $\lim_{n \rightarrow \infty} \sqrt n \tau_n = \infty$. This fact finishes the proof of Theorem~\ref{universality}.
\end{proof}

\appendix
\section{Fuss-Catalan Distribution}\label{Fuss-Catalan Distribution}
For any $m \in \mathbb N$ let us consider the sequence of numbers
$$
M_k = \frac{1}{m k + 1}\binom{k}{mk+k}, \, k \in \mathbb N \cup \{0\}.
$$
These numbers are called Fuss-Catalan numbers. In~\cite{PensZycz2011} the density function $P_m(x)$ which satisfy
$$
\int_0^{K_m} x^k P_m(x) dx = M_k
$$
were found. Here $K_m: = (m+1)^{m+1}/m^m$. The explicit formula for $P_m(x)$ is given by the following formula
\begin{align*}
\label{eq:FCs}
P_m(x) &=
 \sum_{k=1}^m  \Lambda_{k,m}\; x^{\frac{k}{m+1}-1} \times \\
&\qquad \qquad {}_mF_{m-1}\Bigl( \Bigl[ \Bigl\{1-\frac{1+j}{m} +\frac{k}{m+1} \Bigr\}_{j=1}^m \Bigr] , \; \, \\
& \qquad\qquad\Bigl[ \Bigl\{1+\frac{k-j}{m+1} \Bigr\}_{j=1}^{k-1} ,
      \Bigl\{1+\frac{k-j}{m+1} \Bigr\}_{j=k+1}^{m} \Bigr] ; \;
         \frac{m^m}{(m+1)^{m+1}} x
                   \Bigr) \ .
\end{align*}
where the coefficients $\Lambda_{k,m}$ read for $k=1,2, \dots, m$
\begin{equation*}
\label{eq:FCs2}
\Lambda_{k,m} := m^{-3/2} \sqrt{\frac{m+1}{2\pi}}
 \Bigl(\frac{m^{m/(m+1)}}{m+1}\Bigr)^{k}  \
 \frac{\Bigl[ \prod_{j=1}^{k-1} \Gamma\bigl(\frac{j-k}{m+1}\bigr) \Bigr]
       \Bigl[ \prod_{j=k+1}^{m} \Gamma\bigl(\frac{j-k}{m+1}\bigr) \Bigr] }
  { \prod_{j=1}^{m} \Gamma\bigl( \frac{j+1}{m} - \frac{k}{m+1}\bigr) } \ .
\end{equation*}
For example,
\begin{equation*}
P_1(x) = \frac{\sqrt{1-x/4}}{\pi \sqrt{x}}.
\end{equation*}
and
\begin{equation*}
\!\!
P_2(x) =  \frac{\sqrt[3]{2} \sqrt{3}}{12 \pi} \;
 \frac{\bigl[\sqrt[3]{2} \left(27 + 3\sqrt{81-12x} \right)^{\frac{2}{3}} -
   6\sqrt[3]{x}\bigr] } {x^{\frac{2}{3}}
     \left(27 + 3\sqrt{81-12x} \right)^{\frac{1}{3}}},
\label{pi2}
\end{equation*}
valid for $ x \in [0,27/4]$.


\section{Unitary matrix derivatives}
In this section we collect usefull facts about matrix derivatives and matrix exponent.
Let us consider a function $f(\lambda)$ and denote by
$$
\hat f(t) = \frac{1}{2 \pi}\int_{-\infty}^{\infty} f(\lambda) e^{- i t \lambda} \,d\lambda
$$
its Fourier transform. Function $f(\lambda)$ may be reconstucted from $\hat f(t)$ via inverse Fourier transfrom
$$
f(\lambda) = \int_{-\infty}^{\infty} \hat f(t) e^{i t \lambda} \, dt.
$$
Let $f^{(k)}(\lambda)$ be $k$-th derivative of $f(\lambda)$. Then
\begin{equation}\label{eq: k deri of f}
f^{(k)}(\lambda) = i^k \int_{-\infty}^{\infty} t^k \hat f(t) e^{i t \lambda} \, dt.
\end{equation}
We introduce further notations. Let $\M$ an arbitrary symmetric matrix and $M_{jk}$ be its entries. We denote
\begin{align*}
&D_{jk}: = \partial / \partial M_{jk}; \\
&\U(t): = e^{i t \M}, \, U_{jk}(t) = (\U(t))_{jk}.
\end{align*}
Then we may write
\begin{equation}\label{eq: f(M) representation}
f(\M) = \int_{-\infty}^{\infty} \hat f(t) \U(t) \, dt.
\end{equation}
We will use the following formula
\begin{equation}\label{eq: Duhamel's formula}
e^{(\M_1 + \M_2)t} = e^{\M_1 t} + \int_0^t e^{\M_1(t-s)} \, \M_2 \, e^{(\M_1 + \M_2)s} \, ds,
\end{equation}
valid for arbitrary matrices $\M_1, \M_2$ and $t \in \R$.

In what follows we shall use matrix notation~\eqref{matrix H} and~\eqref{matrix V}.
Consider the singular value decomposition of the matrix $\Y$ of dimension $n \times n$. Let $\LL$ and $\HH$ be unitary matrices of dimension $n \times n$.
Let $\Lam$ be a diagonal matrix whose entries are the singular values of the matrix $\Y$. We have the following representation
$$
\Y = \LL \Lam \HH^{*}
$$
We introduce the following matrix
\begin{align*}
\Z^{*} =\frac{1}{\sqrt 2} \left [\begin{matrix} \LL^{*} & \HH^{*} \\
\LL^{*} &  -\HH^{*} \end{matrix} \right ]
\end{align*}
It is straightforward to check that
\begin{align*}
\Z^{*} \V \Z =\left [\begin{matrix}\Lam & \OO \\
\OO &  -\Lam \end{matrix} \right ]
\end{align*}
and
\begin{align*}
\Z^{*} \U(s) \Z = \Z^{*} \left [\begin{matrix}\Lam(s) & \OO \\
\OO &  \Lam(-s) \end{matrix} \right ] \Z,
\end{align*}
where $\Lam(s)$ is a diagonal matrix such that $[\Lam(s)]_{jj} = e^{[i s \Lambda_{jj}]}, j = 1, .. , n$. A simple calculation yields that
\begin{align} \label{eq: U matrix representation}
\U(s) = \left [\begin{matrix}\U_1 & \U_2 \\
\U_3 &  \U_4 \end{matrix} \right ] = \left [\begin{matrix} \LL(\Lam(s) + \Lam(-s)) \LL^{*} & \LL (\Lam(s) - \Lam(-s)) \HH^{*} \\
\HH (\Lam(s) - \Lam(-s)) \LL^{*} &  \HH (\Lam(s) + \Lam(-s)) \HH^{*} \end{matrix} \right ],
\end{align}

We also denote by $\M^{(j,k)}$ a matrix $\M$ with $M_{jk}$ removed. To calculate derivatives of $\U(s)$ we need the following lemma.
\begin{lemma}\label{L: V derivative}
Let $1 \le j,k \le n$ and $ m \geq 2$. Then
\begin{align}\label{eq: V derivative 1}
\left [\frac{\partial \hat \V}{\partial Y_{jk}^{(1)}} \right ]_{a,b} &= \frac{1}{\sqrt n} [\V_{[1,m-1]}]_{a, k+n}\mathbb I(b = j) + \frac{1}{\sqrt n} [\V_{[2,m]} \J]_{k, b} \mathbb I(a = j)
\end{align}
for any $1 \le a, b \le 2n$.
\end{lemma}
\begin{proof}
We decompose $\hat \V$ in the following way
$$
\hat \V= \left ((\HH^{(jk)})^{(1)} + \frac{Y_{jk}^{(1)}}{\sqrt n} \EE_{j, k} \right )\V_{[2,m-1]} \left ((\HH^{(jk)})^{(m)} + \frac{Y_{jk}^{(1)}}{\sqrt n} \EE_{k+n, j+n} \right ) \J.
$$
It is easy to see
$$
[\V_{[1,m-1]} \EE_{k+n, j+n} \J ]_{a, b} =  [\V_{[1,m-1]}]_{ a , k+n} \mathbb I(b = j),
$$
$$
[\EE_{j, k} \V_{[2,m]}  \J ]_{a, b} = [\V_{[2,m]}  \J ]_{k, b} \mathbb I(a = j)
$$
and
$$
[\EE_{j, k} \V_{[2,m-1]}  \EE_{k+n, j+n}   \J ]_{a, b} = [\V_{[2,m-1]}]_{k, k+n} \mathbb I(a = b = j) = 0.
$$
\end{proof}

We may generalize the last lemma on the case when the derivative is taken with respect to $Y_{jk}^{(q)}, q = 2, ... , m$. We have the following lemma
\begin{lemma}\label{L: V derivative with q geq 2}
Let $1 \le j,k \le n$ and $ m \geq 2$. Then
\begin{align}\label{eq: V derivative 1 with q geq 2}
\left [\frac{\partial \hat \V}{\partial Y_{jk}^{(q)}} \right ]_{a,b} &= \frac{1}{\sqrt n} [\V_{[1,m-q]}]_{a, k+n} [\V_{[m-q+2,m]} \J]_{j+n, b}\\
&+\frac{1}{\sqrt n}[\V_{[1,q-1]}]_{a, j} [\V_{[q+1,m]} \J]_{k, b}. \nonumber
\end{align}
\end{lemma}
\begin{proof}
The proof is similar.
\end{proof}

\begin{lemma}\label{l: U derivative}
Let $1 \le j,k \le n$ and $ m \geq 2$. Then
\begin{align*}
\left [\frac{\partial \U(t)}{\partial Y_{jk}^{(1)}} \right ]_{x,y}  &=  \frac{i}{\sqrt n} [\U \V_{[1,m-1]} ]_{x,k+n}*[\U]_{y, j} (t) \\
&+\frac{i}{\sqrt n} [\U \V_{[1,m-1]}]_{y,k+n}*[\U]_{x, j} (t).
\end{align*}
\end{lemma}
\begin{proof}
Using the chain rule we will have
$$
\frac{\partial \U(t)}{\partial Y_{jk}^{(1)}} = \sum_{a = 1}^{n} \sum_{b = n+1}^{2n} \frac{\partial \U(t)}{\partial \hat V_{ab}} \frac{\partial \hat \V_{a,b}}{\partial Y_{jk}^{(1)}}
+ \sum_{a = n+1}^{2n}\sum_{b = 1}^{n}  \frac{\partial \U(t)}{\partial \hat V_{ab}} \frac{\partial \hat \V_{a,b}}{\partial Y_{jk}^{(1)}}.
$$
Applying Lemma~\ref{L: V derivative} we will have
\begin{align*}
\frac{\partial \U(t)}{\partial Y_{jk}^{(1)}} &= \frac{1}{\sqrt n} \sum_{b = n+1}^{2n} \frac{\partial \U(t)}{\partial \hat V_{jb}} [\V_{[2,m]}\J]_{ k , b} + \frac{1}{\sqrt n} \sum_{a = n+1}^{2n} \frac{\partial \U(t)}{\partial \hat V_{aj}} [\V_{[1,m-1]}]_{a, k+n}.
\end{align*}
From~\eqref{eq: Duhamel's formula} it follows that
\begin{align*}
\frac{\partial \U(t)}{\partial Y_{jk}^{(1)}} &= \frac{i}{\sqrt n} \sum_{b = n+1}^{2n} U_{xj}*U_{by}(t) [\V_{[2,m]}\J]_{ k , b} \\
&+ \frac{i}{\sqrt n} \sum_{a = n+1}^{2n} U_{x,a}*U_{jy}(t) [\V_{[1,m-1]}]_{a, k+n}.
\end{align*}
Since $[\U \V_{[1,m-1]}]_{ y , k+n} = [\V_{[2,m]} \J \U ]_{k, y}$ we get the statement of Lemma.
\end{proof}

\begin{lemma}\label{l: U derivative q geq 2}
Let $1 \le j,k \le n$ and $ m \geq 2$. Then
\begin{align*}
\left [\frac{\partial \U(t)}{\partial Y_{jk}^{(q)}} \right ]_{x,y}  &=  \frac{i}{\sqrt n} [\U \V_{1,[m-q]}]_{x,k+n}*[\V_{[m-q+2,m]}\J\U]_{j+n,y} (t) \\
&+\frac{i}{\sqrt n} [\U \V_{1,[m-q]} ]_{y,k+n}*[\V_{[m-q+2,m]}\J\U]_{j+n,x} (t).
\end{align*}
\end{lemma}
\begin{proof}
The proof is similar.
\end{proof}
The following lemma gives an expression for derivative of $S(\hat \V): = \frac{1}{2} \Tr f(\hat \V)$ with respect to $Y_{jk}^{(1)}$.

\begin{lemma}\label{l: S derivative}
Let $1 \le j,k \le n$ and $ m \geq 2$. Then
\begin{align}\label{eq: S derivative 1}
\frac{\partial S}{\partial Y_{jk}^{(1)}} =  \frac{1}{\sqrt n}[f'(\hat \V) \V_{[1,m-1]}]_{j, k+n}.
\end{align}
\end{lemma}
\begin{proof}
It is easy to see that
$$
\frac{\partial S}{\partial Y_{jk}^{(1)}} = \frac{1}{2} \int_{-\infty}^{\infty} \hat f(u) \Tr \frac{\partial \U(u)}{\partial Y_{jk}^{(1)}}\, du.
$$
Applying Lemma~\ref{l: U derivative} we get
\begin{align*}
\frac{\partial S}{\partial Y_{jk}^{(1)}} &=  \frac{i}{ 2\sqrt n} \int_{-\infty}^{\infty} s \hat f(s) [\U(s) \V_{[1,m-1]}]_{j, k+n} \, ds \\
& + \frac{i}{2\sqrt n} \int_{-\infty}^{\infty} s \hat f(s) [\V_{[2,m]} \J \U(s)]_{k, j} \, ds
\end{align*}
Applying the properties of $\V$ and $\U$ we get~\eqref{eq: S derivative 1}.
\end{proof}

\begin{lemma}\label{l: S derivative q geq 2}
Let $1 \le j,k \le n$ and $ m \geq 2$. Then
\begin{align}\label{eq: S derivative 1 q geq 2}
\frac{\partial S}{\partial Y_{jk}^{(q)}} =  \frac{1}{\sqrt n}[\V_{[m-q+2, m]} \J f'(\hat \V) \V_{[1,m-q]}]_{j+n, k+n}.
\end{align}
\end{lemma}
\begin{proof}
The proof is similar.
\end{proof}

\section{Auxiliary lemmas}
In this section we prove some auxiliary Lemmas.
%

The following Lemma gives an estimate for variance of
$$
T_n(s,t): = \frac{1}{n}\sum_{j,k=1}^n [\HH^{(2)} \J \U(s)]_{k,j} [\U(t-s)\HH^{(1)}]_{j,k+n}.
$$
\begin{lemma}\label{l: second trace variance}
Under condition of Theorem~\ref{th:main} we have
\begin{equation*}
\Var(T_n(t,s)) \le \frac{C \max(t^2, (t-s)^2)}{n}.
\end{equation*}
\end{lemma}
\begin{proof}
Let us introduce the following matrices
$$
\HH^{(q,l)} = \HH^{(q)} - \ee_l \ee_l^T \HH^{(q)} - \HH^{(q)} \ee_l \ee_l^T
$$
$$
\widetilde \HH^{(q,l)} = \HH^{(q)} - \ee_{l+n} \ee_{l+n}^T \HH^{(q)} - \HH^{(q)} \ee_{l+n} \ee_{l+n}^T
$$
where $q = 1,2$ and $l = 1, ... , n$.
We define the following filtration
$$
\mathcal{F}_{q,l} = \sigma\{Y_{i_1,i_2}^{(q)}, l < i_1, i_2 \le n, Y_{i_3i_4}^{(2)}, i_3, i_4 = 1, ... ,n \}.
$$
We may rewrite the difference
\begin{align*}
&\E \sum_{j,k=1}^n [\HH^{(2)} \J \U(s)]_{k,j} [\U(t-s) \HH^{(1)}]_{j,k+n} - \sum_{k=1}^n [\HH^{(2)} \J \U(s)]_{k,j} [\U(t-s)\HH^{(1)}]_{j,k+n} \\
&=\sum_{q=1}^2 \sum_{l=1}^n (\E_{q,l} - \E_{q, l-1}),
\end{align*}
where $\E_{q,l}$ is the mathematical expectation with respect to $\mathcal F_{q,l}$. It is easy to see
that $\mathcal F_{1,n} = \mathcal F_{2,0}$ and
\begin{align*}
&\E_{1,l} \sum_{j, k=1}^n [\widetilde \HH^{(2,l)} \J \U^{(1,l)}(s)]_{k,j}[\U^{(1,l)}(t-s) \HH^{(1,l)}]_{j,k+n} \\
&\qquad\qquad\qquad= \E_{1,l-1} \sum_{j,k=1}^n [\widetilde \HH^{(2,l)} \J \U^{(1,l)}(s)]_{k,j} [\U^{(1,l)}(t-s)\HH^{(1,l)}]_{j,k+n} \\
&\E_{2,l} \sum_{j,k=1}^n [\HH^{(2,l)} \J \U^{(2,l)}(s)]_{k,j} [\U^{(2,l)}(t-s)\widetilde \HH^{(1,l)}]_{j,k+n} \\
&\qquad\qquad\qquad= \E_{2,l-1} \sum_{j,k=1}^n [\HH^{(2,l)} \J \U^{(2,l)}(s)]_{k,j}[\U^{(2,l)}(t-s) \widetilde \HH^{(1,l)}]_{j,k+n}
\end{align*}
We consider the case $q = 1$ only. The case $q=2$ is similar.
We may write
\begin{align*}
&\sum_{j,k=1}^n [\HH^{(2)} \J \U(s)]_{k,j} [\U(t-s)\HH^{(1)}]_{j,k+n} \\
&- \sum_{j,k=1}^n [\widetilde \HH^{(2,l)} \J ]_{k,j} [\U^{(1,l)}(t-s)\HH^{(1,l)}]_{j,k+n} =
\Theta_{1,l} + \Theta_{2,l} + \Theta_{3,l}+\Theta_{4,l},
\end{align*}
where
\begin{align*}
&\Theta_{1,l}  = \sum_{j,k=1}^n [(\HH^{(2)} - \widetilde \HH^{(2,l)}) \J \U(s)]_{k,j}[\U(t-s) \HH^{(1)}]_{j,k+n},\\
&\Theta_{2,l}  = \sum_{j,k=1}^n [\widetilde \HH^{(2,l)} \J (\U(s)-\U^{(1,l)}(s))]_{j,k} [\U(t-s) \HH^{(1)}]_{j,k+n},\\
&\Theta_{3,l}  = \sum_{j,k=1}^n [\widetilde \HH^{(2,l)} \J \U^{(1,l)}(s)]_{j,k} [(\U(t-s)-\U^{(1,l)}(t-s)) \HH^{(1)}]_{j,k+n},\\
&\Theta_{4,l}  = \sum_{j,k=1}^n [\widetilde \HH^{(2,l)} \J \U^{(1,l)}(s)] [\U^{(1,l)}(t-s)(\HH^{(1)} - \HH^{(1,l)})]_{k,k+n}.
\end{align*}
It is easy to check that $\Theta_{1,l} = \Theta_{4,l} = 0$.
We consider the term $\Theta_{2,l}$. The term $\Theta_{3,l}$ is similar.
Applying~\eqref{eq: Duhamel's formula} we get
$$
\Theta_{2,l} = I_{1,l} + I_{2,l},
$$
where
\begin{align*}
&I_{1,l} = \int_0^s \sum_{j,k=1}^n [\widetilde \HH^{(2,l)} \J \U^{(1,l)}(s_1)\ee_l \ee_{l}^T \V \J \U(s-s_1)]_{k,j} [\U(t-s)\HH^{(1)}]_{j,k+n} \, ds_1 ,\\
&I_{2,l} = \int_0^s \sum_{j,k=1}^n [\widetilde \HH^{(2,l)} \J \U^{(1,l)}(s_1)\V \J \ee_l \ee_{l}^T \U(s-s_1)]_{k,j} [\U(t-s)\HH^{(1)}]_{j,k+n} \, ds_1 ,\\
\end{align*}
Doing simple calculations we get
$$
I_{1,l} = \int_0^s [\W \U_3(s-s_1) \U_2(t-s)(\Y^{(2)})^T \Y^{(2)} \U_3^{(1,l)}(s_1)]_{ll} \, ds_1,
$$
It is easy to derive the following estimate
$$
\sum_{l=1}^n \E I_{1,l}^2 \le C s^2 \E||\W (\Y^{(2)})^T \Y^{(2)}||_2^2 \le C s^2 n.
$$
The same is true for $\sum_{l=1}^n \E I_{2,l}^2$. This fact finishes the proof of Lemma.
\end{proof}
The following lemma gives an estimate for the variance of $\frac{1}{n} u_n(t)$, $\V_{n,j}(t)$ and
$T_j(t):=[\HH^{(2)} \J \U(t) \HH^{(1)}]_{j,j+n}$ for all $j = 1, ... , n$.
\begin{lemma}\label{l: variance of trace general case}
Under condition of Theorem~\ref{th:main general case} we have
\begin{align}\label{eq: var 1 general case}
&\Var\left[\frac{1}{n}u_n(t)\right] \le \frac{C}{n}, \\
\label{eq: var 2 general case}
&\Var[V_{n,j}(t)] = o(1) \quad j = 1, ... , n; \\
\label{eq: var 3 general case}
&\Var[T_j(t)]  = o(1) \quad  j = 1, ... , n.
\end{align}
\end{lemma}
\begin{proof}
The proof of the first statement~\eqref{eq: var 1 general case} may be realized similarly to the proof of Lemma~\ref{l: second trace variance}. One may also use the result for the matrix resolvent and the Stieltjes transform. We present the proof of~\eqref{eq: var 3 general case} only. The proof of~\eqref{eq: var 2 general case} is similar. Let us denote
$$
K_{j,n}(t_1, t_2) = \E [T_j(t_1)(T_j(t_2) - \E T_j(t_2))] = \E T_j(t_1) T_j^0(t_2),
$$
where $T_j^0(t) : = T_j(t) - \E T_j(t)$. We have
$$
K_{j,n}(t_1, t_2) = \frac{1}{\sqrt n} \sum_{k=1}^n Y_{jk}^{(2)} [\U(t_1) \HH^{(2)}]_{k+n,j+n} T_j^0(t_2)
$$
By Taylor's formula
\begin{align*}
&K_{j,n}(t_1, t_2) = \frac{1}{\sqrt n} \sum_{k=1}^n \E \left[\frac{\partial \U(t_1)}{\partial Y_{jk}^{(2)}} \HH^{(1)}\right]_{k+n,j+n} T_j^0(t_2) \\
&+\frac{1}{\sqrt n} \sum_{k=1}^n\E \left[\U(t_1)\frac{\partial \HH^{(1)}}{\partial Y_{jk}^{(2)}} \right]_{k+n,j+n} T_j^0(t_2)\\
&+\frac{1}{\sqrt n} \sum_{k=1}^n \E[ \U(t_1)\HH^{(1)}]_{k+n,j+n} \left [\frac{\partial \HH^{(2)}}{\partial Y_{jk}^{(2)}} \J \U(t_2) \HH^{(1)} \right ]_{j,j+n} \\
&+ \frac{1}{\sqrt n} \sum_{k=1}^n \E[ \U(t_1)\HH^{(1)}]_{k+n,j+n} \left [\HH^{(2)}\J\frac{\partial \U(t_2)}{\partial Y_{jk}^{(2)}} \HH^{(1)} \right ]_{j,j+n}\\
&+\frac{1}{\sqrt n} \sum_{k=1}^n \E[ \U(t_1)\HH^{(1)}]_{k+n,j+n} \left [\HH^{(2)}\J \U(t_2)\frac{\partial \HH^{(1)}}{\partial Y_{jk}^{(2)}}\right ]_{j,j+n} \\
&+R,
\end{align*}
where $R$ is a remainder term. It is straightforward to check $R$ has the order $o(1)$. By Lemma~\ref{l: U derivative} we get
\begin{align*}
&K_{j,n}(t_1, t_2) = \frac{i}{n} \E u_n* [\HH^{(2)} \J \U(t_1) \HH^{(1)} ]_{j+n, j+n} T_j^0(t_2) \\
&+\frac{i}{n} \sum_{k=1}^n \E[\U \HH^{(1)}]_{k+n, j+n}*[\U(t_1) \HH^{(1)}]_{k+n, j+n}  T_j^0(t_2) \\
&+\frac{1}{n} \E u_n(t_1) T_j^0(t_2) \\
&+\frac{2}{n} \sum_{k=1}^n \E[ \U(t_1)\HH^{(1)}]_{k+n,j+n} [\U(t_2) \HH^{(1)} ]_{k+n,j+n} \\
&+\frac{2i}{n} \sum_{k=1}^n \E [ \U(t_1)\HH^{(1)}]_{k+n,j+n} [\HH^{(2)}\J \U]_{j, k+n}*[\HH^{(2)} \J \U(t_2) \HH^{(1)}]_{j+n,j+n}\\
&+R.
\end{align*}
Similarly to the previous estimates it is not very difficult to check that all term except the first one have the order $o(1)$. Let us consider the first term
\begin{align*}
&\frac{i}{n} \E u_n* [\HH^{(2)} \J \U(t_1) \HH^{(1)} ]_{j+n, j+n} T_j^0(t_2) \\
&\qquad\qquad\qquad= \frac{i}{n}\E u_n* \E[\HH^{(2)} \J \U(t_1) \HH^{(1)} ]_{j+n, j+n} T_j^0(t_2) + o(1).
\end{align*}
From~\eqref{eq: V n k equation} we have
$$
V_{n,k}(s) = \frac{i}{n}\E u_n(s)*\E T_k(s) + o(1).
$$
We may conclude that
\begin{align*}
&\frac{i}{n} \sum_{k=1}^n \E[\U]_{k+n, k+n}* \E[\HH^{(2)} \J \U(t_1) \HH^{(1)} ]_{j+n, j+n} T_j^0(t_2) \\
&\qquad\qquad\qquad\qquad= -\frac{1}{n^2}(\E u_n)^{*2}*\E T_j(t_1) T_j^0(t_2) + o(1).
\end{align*}
Taking the limit with respect to $n_l \rightarrow \infty$ we get that $K_j: = \lim_{n_l \rightarrow \infty} K_{j,n_l}$ satisfies the following equation
$$
K_j(t_1,t_2) = - \int_0^{t_1} v^{*2}(t_1 - s) K_j(s, t_2) \, ds.
$$
Since $K_j(t_1, t_2) = 0 $ is a unique solution of the last equation this means that
$$
K_{j,n}(t_1, t_2) = o(1).
$$
One may take $t_2 = t_1$ and finish the proof of Lemma.
\end{proof}

\section{Laplace transform} \label{ap: laplace transform}
In this section we recall several results from the theory of Laplace transform. We will follow~\cite{LytPastur2009}[Proposition~2.1].
\begin{statement}\label{stat: Laplace transform}
Let $f: \R_{+} \rightarrow \C$ be locally Lipshitzian and such that for some $\delta > 0$
$$
\sup_{t \geq 0} e^{-\delta t} |f(t)| < \infty
$$
and let $\tilde f: \{z \in C: \imag z < -\delta\} \rightarrow \C$ be its generalized Fourier transform
$$
\tilde f(z) = \frac{1}{i} \int_0^\infty e^{-izt} f(t) \, dt.
$$
The inversion formula is given by
$$
f(t) = \frac{i}{2\pi} \int_L e^{izt} \tilde f(z) \, dz, \quad t\geq 0,
$$
where $L = (-\infty - i \varepsilon, \infty - i \varepsilon ), \varepsilon > \delta$, and the principal value of the integral at infinity is used. Denote the correspondence between functions and
their generalized Fourier transforms as $f \leftrightarrow \tilde f$. Then we have
\begin{align*}
&f'(t)  \leftrightarrow i(f(+0) + z \tilde f(z));\\
&\int_0^t f(s) ds  \leftrightarrow (iz)^{-1} \tilde f(z);\\
&f*g(t)  \leftrightarrow i \tilde f(z) \tilde g(z).
\end{align*}
\end{statement}

We calculate the Fourier transforms of some functions.

\begin{lemma}\label{l: FT of K}
Let $s(z)$ be the Stieltjes transform of $p(x)$ which is a symmetrization of Fuss-Catalan density $P_2(x)$, see
Appendix~\ref{Fuss-Catalan Distribution}. The inverse Fourier transform of
$$
K(z) = \frac{1/z-2s(z)}{1-3s^2(z)}
$$
is given by
$$
T(t) =  \frac{1}{\pi}\int_{-a}^a \frac{ e^{itx}}{3p_1(x)}\frac{4p_1^4(x) + 11 p_1^2(x) + 4}{4 p_1^2(x) + 3} \, dx,
$$
where $p_1(x) = \pi p(x)$.
\end{lemma}
\begin{proof}
Be definition, see Statement~\ref{stat: Laplace transform},
$$
T(t) =\frac{i}{2\pi}\int_{L} e^{itz}K(z)\,dz,
$$
where $L = (-\infty - i \varepsilon, \infty - i \varepsilon )$.
We introduce the following contour $\mathbf K$ (see Figure~\ref{fig:contour}) :
\begin{equation*}
\mathbf K:=\mathbf K_1\cup\cdots\cup\mathbf K_8,
\end{equation*}
where
\begin{align*}
&\mathbf K_1:=\{z=u+iv, |u|\le T, v=-\varepsilon\},
\quad\mathbf K_2:=\{z=u+iv: |z|=T, v \ge 0 \},\\
&\mathbf K_{3,4}:=\{z=u+iv: |u|\le a + \varepsilon/2, v=\pm \varepsilon/2\},\\
&\mathbf K_{5,6}:=\{z=u+iv: u=\pm (a+\varepsilon/2), -\varepsilon/2\le v\le\varepsilon/2\},\\
&\mathbf K_{7,8}:=\{z=u+iv: u=\pm T, -\varepsilon\le v\le 0\}.
\end{align*}

\begin{figure}
\begin{center}
\begin{tikzpicture}
\draw [help lines,->] (-4, 0) -- (4.6,0);
\draw [help lines,->] (0, -2) -- (0, 4);
\draw[thick,black,xshift=2pt,
decoration={ markings,  
      mark=at position 0.2 with {\arrow{latex}},
      mark=at position 0.2 with {\arrow{latex}},
      mark=at position 0.8 with {\arrow{latex}},
      mark=at position 0.8 with {\arrow{latex}}},
      postaction={decorate}]
  (3,0)  arc (0:180:3) -- (-3,0);
\draw[thick,black,xshift=2pt,
decoration={ markings,  
      mark=at position 0.2 with {\arrow{latex}},
      mark=at position 0.2 with {\arrow{latex}},
      mark=at position 0.8 with {\arrow{latex}},
      mark=at position 0.8 with {\arrow{latex}}},
      postaction={decorate}]
  (-2,-0.5) -- (-2,0.5) -- (2,0.5)--(2,-0.5) -- (-2, -0.5) ;
\draw[thick,black,xshift=2pt,
decoration={ markings,  
      mark=at position 0.2 with {\arrow{latex}},
      mark=at position 0.2 with {\arrow{latex}},
      mark=at position 0.8 with {\arrow{latex}},
      mark=at position 0.8 with {\arrow{latex}}},
      postaction={decorate}]
  (-3,0) -- (-3,-1) -- (3,-1)--(3,0) ;
\node at (4.5,-0.2){$x$};
\node at (4.5,-0.2){$x$};
\node at (-0.24,3.8) {$iy$};
\node at (-0.6,0.8) {$K_3$};
\node at (-0.6,-0.8) {$K_4$};
\node at (-1.8,2.8) {$K_2$};
\node at (2.4,0.25) {$K_6$};
\node at (-2.3,0.25) {$K_5$};
\node at (1.555, -1.4) {$K_1$};
\node at (3.4, -0.5) {$K_7$};
\node at (-3.3, -0.5) {$K_8$};
\end{tikzpicture}
\caption{Contour of integration}
\label{fig:contour}
\end{center}
\end{figure}
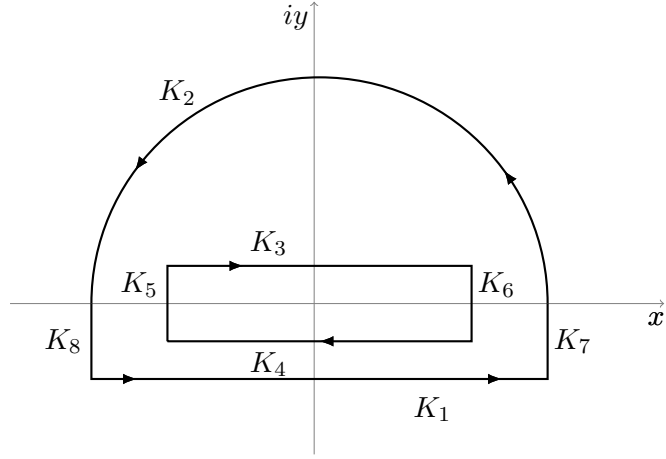

We may write
\begin{equation*}
T(t)=\lim_{T\to\infty}\frac{i}{2 \pi} \int_{\mathbf K_1}e^{itz}K(z)\,dz
\end{equation*}
and
\begin{equation*}
\int_{\mathbf K} e^{itz} K(z)\, dz=0.
\end{equation*}
Furthermore, we note
\begin{equation*}
\lim_{T\rightarrow\infty}\int_{\mathbf K_2 \cup \mathbf K_7 \cup \mathbf K_8} e^{itz} K(z)\, dz = 0.
\end{equation*}
We compute the integrals
\begin{equation*}
\mathcal K_1:=\left (\int_{\mathbf K_3}-\int_{\mathbf K_4} \right )e^{itz} K(z) \,dz.
\end{equation*}
Let $s(z)=if(z)+g(z)$, for $z=u+iv$ with $v>0$. Note that by definition
\begin{equation*}
\imag s(z)=\begin{cases}
f(z),\text{ if } \imag z>0,\\
-f(z),\text{ if } \imag z<0.
\end{cases}
\end{equation*}
Let us calculate $K(z)$ for $z \in K_3$. Applying
$$
1 + z s(z) = z s^3(z)
$$
we get
\begin{align*}
K(z) &= \frac{s(z)(s^2(z) - 3)}{1 - 3 s^2(z)} = \frac{1}{3}\frac{(g + if)(f^2 + 1 - 3 f^2 + 6ifg - 9)}{2 p (f - 3 i g)}\\
& = \frac{1}{6f}\frac{(g+if)(6ifg - 2f^2 - 8)}{f - 3 i g} =  \frac{1}{6f}\frac{(g+if)(6ifg - 2f^2 - 8)(f+3ig)}{|f - 3 i g|^2}.
\end{align*}
The enumerator is equal to
\begin{align*}
&(g+if)(6ifg - 2f^2 - 8)(f+3ig) = (6 i f g^2 - 2 f^2 g - 8 g - 6 f^2 g - 2 i f^3 - 8 if)(f+3ig)\\
& = (2 i f + 2i f^3 - 8 f^2 g - 8 g - 2 i f^3 - 8 if)(f+3ig) \\
& = -2(3 i f + 4 f^2 g + 4 g)(f+3 i g).
\end{align*}
The imaginary part of the enumerator is given by
$$
- 6 f^2 -  24 f^2 g^2 - 24 g^2 = - 6 f^2 - 8 f^2 - 8 f^4 - 8 - 8f^2 = -2 (4f^4 + 11 f^2 + 4).
$$
Finally
$$
\imag K(z) = - \frac{1}{3f} \frac{4f^4 + 11 f^2 + 4}{4 f^2 + 3}.
$$
The real part is equal to
$$
\re K(z) = -\frac{1}{3}\frac{g( 5 - 4 f^2)}{4 f^2 + 3}.
$$
It is easy to see that for $z \in K_4$ we will have
$$
\imag K(z) = \frac{1}{3f} \frac{4f^4 + 11 f^2 + 4}{4 f^2 + 3}, \,
\re K(z) = -\frac{1}{3}\frac{g( 5 - 4 f^2)}{4 f^2 + 3}.
$$
Since $\varepsilon$ is an arbitrary number and
$$
\lim_{\varepsilon \rightarrow 0} f(u + i \varepsilon) = \pi p(u)
$$
then integrating $\re K(z)$ in the opposite directions we get zero. Finally
$$
T(t) =  \frac{1}{\pi}\int_{-a}^a \frac{ e^{itx}}{3\pi p(x)}\frac{4(\pi p(x))^4 + 11 (\pi p(x))^2 + 4}{4 (\pi p(x))^2 + 3} \, dx.
$$
\end{proof}

\def\polhk#1{\setbox0=\hbox{#1}{\ooalign{\hidewidth
  \lower1.5ex\hbox{`}\hidewidth\crcr\unhbox0}}}

\end{document}